\documentclass[preprint]{elsarticle}
\usepackage{graphicx}
\usepackage[english,activeacute]{babel}
\usepackage[latin1]{inputenc}
\usepackage{amsmath}
\usepackage{amsfonts}
\usepackage{amssymb,latexsym}
\usepackage[T1]{fontenc}
\usepackage{array}
\usepackage{booktabs}
\usepackage{graphicx}

\newtheorem{theorem}{Theorem}

\newtheorem{corollary}{Corollary}

\newtheorem{lemma}[theorem]{Lemma}

\newtheorem{proposition}{Proposition}
\newtheorem{remark}{Remark}

\newenvironment{proof}[1][Proof]{\noindent\textbf{#1.} }{\ \rule{0.5em}{0.5em}}
\begin{document}

\begin{frontmatter}

\title{On polynomials associated with an Uvarov modification of a quartic potential Freud-like weight.}
\author{Alejandro Arceo$^{1}$, Edmundo J. Huertas$^{2,*}$, Francisco Marcell\'an$^{3}$
\\[2mm]
$^{1}$Departamento de Matem\'aticas\\
Universidad Carlos III de Madrid, Avda. de la Universidad, 30, 28911 Legan\'{e}s, Madrid, Spain.\\
grarceo@gmail.com\\
[3mm] $^{2}$ Departamento de F\'{i}sica y Matem\'{a}ticas\\
Universidad de Alcal\'{a}, Ctra. Madrid-Barcelona, Km. 33.600, 28871 Alcal\'{a} de Henares, Madrid, Spain.\\
edmundo.huertas@uah.es, ehuertasce@gmail.com\\
[3mm] $^{3}$ Instituto de Ciencias Matem\'{a}ticas (ICMAT) y Departamento de Matem\'aticas\\
Universidad Carlos III de Madrid, Avda. de la Universidad, 30, 28911 Legan\'{e}s, Madrid, Spain.\\
pacomarc@ing.uc3m.es}


\begin{abstract}
In this contribution we consider sequences of monic polynomials orthogonal with respect to the standard Freud-like inner product involving a quartic potential
\begin{equation*}
\left\langle p,q\right\rangle _{M}=\int_{\mathbb{R}%
}p(x)q(x)e^{-x^{4}+2tx^{2}}dx+Mp(0)q(0).
\end{equation*}%
We analyze some properties of these polynomials, such as the ladder operators and the holonomic equation that they satisfy and, as an application, we give an electrostatic interpretation of their zero distribution in terms of a logarithmic potential interaction under the action of an external field. It is also shown that the  coefficients of their three term recurrence relation satisfy a nonlinear difference string equation. Finally, an equation of motion for their zeros in terms of their dependence on $t$ is given.
\end{abstract}

\begin{keyword}
Orthogonal polynomials $\circ$ Freud-like
weights $\circ$ Logarithmic potential $\circ$ String equation $\circ$ Semi-classical linear functional.

\smallskip

\MSC 33C45 \sep 33C47

{\footnotesize * Corresponding author.}
\end{keyword}
\end{frontmatter}



\section{Introduction}

\label{[SECTION-1]-Intro}



Let us consider the so called Freud-like inner products%
\begin{equation*}
\left\langle p,q\right\rangle _{t}=\int_{\mathbb{R}}p(x)q(x)d\mu
_{t}(x),\quad p,q\in \mathbb{P},
\end{equation*}%
where $d\mu _{t}(x)=\omega _{t}(x)dx=e^{-V_{t}(x)}dx$ is a positive,
nontrivial Borel measure supported in the whole real line $\mathbb{R}$. The
properties of such sequences of polynomials are known only for certain
values of the external potential $V_{t}(x)$. Thus, a well studied example in
the literature corresponds to the case $V_{t}(x)=x^{4}-2tx^{2}$, which leads
to the inner product%
\begin{equation}
\langle p,q\rangle _{t}=\int_{\mathbb{R}}p(x)q(x)e^{-x^{4}+2tx^{2}}dx,\quad
p,q\in \mathbb{P}.  \label{Fb-InnerProduct}
\end{equation}
Here $t\in \mathbb{R}_{+}$ is a parameter, which, in other contexts, can be
interpreted as the \textquotedblleft time\textquotedblright \- variable. We
denote by $\{F_{n}^{t}\}_{n\geqslant 0}$ the corresponding sequence of monic
orthogonal polynomials (SMOP, in short) which constitutes a family of
semi-classical orthogonal polynomials (see \cite{M-ACAM91}, \cite{S-DMJ39}),
because $V_{t}(x)$ is differentiable in $\mathbb{R}$\ (the support of $\mu
_{t}$), and the weight function satisfies the distributional equation, known
in the literature as Pearson equation (see \cite{V-WS07})%
\begin{equation*}
\lbrack \omega _{t}(x)]^{\prime }=(-4x^{3}+4tx)\omega _{t}(x).
\end{equation*}
where \textquotedblleft $\;{}^{\prime }$ \textquotedblright $\,$ denotes
derivative with respect to the variable $x$.

We must point out that the class of such a family is $s=2$, and that several
SMOP related to (\ref{Fb-InnerProduct}) have been studied in the literature.
For example, for $t=0$ the asymptotic behavior of the corresponding SMOP has
been deeply analyzed in \cite{N-CMS83}. In \cite{BV-PAMS10} a
Riemann-Hilbert approach is done to find several semi-classical asymptotic
results for these polynomials. As an application of the semi-classical
asymptotics of the orthogonal polynomials, the authors prove the
universality of the local distribution of eigenvalues in the matrix model
with the double-well quartic interaction in the presence of two cuts. In 
\cite{DK-N06}\ the authors study the varying quartic Freud-like weight $%
e^{-V(x)}$, with $V(x)=N\left( t\frac{x^{4}}{4}+\frac{x^{2}}{2} \right) $
for $t<0$, where the orthogonality takes place on certain contours of the
complex plane. They use the Deift/Zhou steepest descent analysis of the
Riemann-Hilbert problem associated with the corresponding polynomials, to
present an alternative and a more direct proof of the Fokas, Its and Kitaev
result, showing that there exists a critical value for $t$ around which the
asymptotics of the recurrence coefficients are described in terms of exactly
specified solutions of the Painlevé I equation. In \cite{M-JAT86} the author
explores the nonlinear difference equation satisfied by the coefficients in
the three term recurrence relations for polynomials orthogonal with respect
to exponential weights, and in \cite{M-SIDE99}\ the author finds the
relation of such Freud's equations with discrete Painlevé equations for
certain potentials, as $V(x)=x^{2}$, $V(x)=\alpha x^{4}+\beta x^{2}$\ or $%
V(x)=x^{6}$, among others. The survey \cite{L-AAM87} presents analytic
aspects of general orthogonal polynomials associated with several
exponential weights on finite and infinite intervals. The works \cite%
{DKMZ-CPAM99-1} and \cite{DKMZ-CPAM99-2} deal with uniform and strong
asymptotics for polynomials orthogonal with respect to certain exponential
weights, and their connections with the random matrix theory. Magnus showed
in \cite{M-JCAM95}\ that the recurrence coefficients of semi-classical
orthogonal polynomials are shown to be solutions of nonlinear differential
equations with respect to a well-chosen parameter, according to principles
established by D. Chudnovsky and G. Chudnovsky.


In the linear space $\mathbb{P}$ of polynomials with real coefficients, let
us introduce the following inner product%
\begin{eqnarray}
\langle p,q\rangle &=&\int_{\mathbb{R}}p\left( x\right) q\left( x\right)
\omega _{t}(x)dx+Mp(0)q(0)  \label{Freud-typeInnProd} \\
&=&\left\langle p,q\right\rangle _{t}+Mp(0)q(0),  \notag
\end{eqnarray}%
with $M\in \mathbb{R}_{+}\,$. We denote by $\{Q_{n}^{t}\}_{n\geqslant 0}$
the SMOP with respect to the above inner product. $\{Q_{n}^{t}\}_{n\geqslant
0}$ is said to be the sequence of monic polynomials orthogonal with respect
to the measure $d\nu (x)=d\mu _{t}(x)+M\,\delta _{0}(x),$ where $\delta
_{0}(x)$ is the Dirac delta function at $x=0$. These polynomials are known
as \textit{Freud-type} or \textit{Freud-Krall} orthogonal polynomials. It is
clear that $\{Q_{n}^{t}\}_{n\geqslant 0}$ is a standard sequence, since the
operator of multiplication by $x$ is symmetric with respect to such an inner
product, i.e. $\langle xp,q\rangle =\langle p,xq\rangle $, for every $p,q\in 
\mathbb{P}$.

This inner product was first studied in \cite{GAM-ETNA05} in the particular
case $t=0$. There, the authors introduce a Dirac mass point at $x=0$ and
provide the holonomic second order linear differential equation with varying
polynomial coefficients. They also give an electrostatic interpretation for
the distribution of zeros of the corresponding orthogonal polynomials.

Following this initial work, the aim of our contribution is twofold. First,
we generalize the electrostatic interpretation of the zero distribution in
terms of a logarithmic potential interaction under the action of an external
field for values of $t$ greater than zero. Second, we derive an equation of
motion for the distribution of the zeros of the corresponding Freud-type
orthogonal polynomials.

The remainder of this Section will be devoted to provide some basic
structural properties of $\{F_{n}^{t}\}_{n\geqslant 0}$ to be used in the
sequel.



\begin{proposition}
\label{FbMOPS}Let $\{F_{n}^{t}\}_{n\geqslant 0}$ denote the sequence of
Freud monic polynomials orthogonal with respect to (\ref{Fb-InnerProduct}).
The following statements hold.

\begin{enumerate}
\item Three term recurrence relation. For every $n\in \mathbb{N}$, the
recurrence relation for $\{F_{n}^{t}\}_{n\geqslant 0}$, is given by (see 
\cite{FVZ-JPA12})%
\begin{equation}
xF_{n}^{t}(x)=F_{n+1}^{t}(x)+a_{n}^{2}(t)F_{n-1}^{t}(x)  \label{3TRR-Monic}
\end{equation}%
with $F_{-1}^{t}(x)=0$ and $F_{0}^{t}(x)=1$. An important feature of these
polynomials is that the recurrence coefficients $a_{n}(t)$\ in the above
three term recurrence relation satisfy the following nonlinear difference
equation%
\begin{equation}
4a_{n}^{2}(t)\left( a_{n-1}^{2}(t)+a_{n}^{2}(t)+a_{n+1}^{2}(t)-t\right)
=n,\quad n\geqslant 1,  \label{StringEq0}
\end{equation}%
with $a_{0}^{2}(t)=0$ and $a_{1}^{2}(t)=||x||_{t}^{2}/||1||_{t}^{2}$. This
is known in the literature as the \textit{string equation} or \textit{Freud
equation }(see \cite{FVZ-JPA12}, \cite[(3.2.20)]{Ism05}, among others).
Moreover, notice that $\omega _{t}(x)$\ is an even weight, i.e. $%
\{F_{n}^{t}\}_{n\geqslant 0}$ is a symmetric sequence of orthogonal
polynomials.

\item We will denote by%
\begin{equation*}
||F_{n}^{t}||_{t}^{2}=\left\langle F_{n}^{t},F_{n}^{t}\right\rangle
_{t}=\int_{\mathbb{R}}[F_{n}^{t}(x)]^{2}e^{-x^{4}+2tx^{2}}dx.
\end{equation*}%
the corresponding norm. Hence%
\begin{equation*}
\zeta _{n}(t)=a_{1}^{2}(t)\cdots a_{n}^{2}(t)>0,
\end{equation*}%
with $||F_{n}^{t}||_{t}^{2}=\zeta _{n}(t)\,||1||_{t}^{2}$.

\item Evolution equation. The recurrence coefficients $a_{n}(t)$\ in (\ref%
{3TRR-Monic}) satisfy%
\begin{equation*}
\dot{a}_{k}(t)=a_{k}(t)\,\left( a_{k+1}^{2}(t)-a_{k-1}^{2}(t)\right) ,\quad
k\geqslant 1,
\end{equation*}%
where \textquotedblleft dot\textquotedblright\ denotes the derivative with
respect to time variable $t>0$\ (see \cite[(1.12)]{ABM-JCAM97},\cite[(2.4)]%
{M-AM75}).

\item Structure relation (\cite[Th. 3.2.1, and p.57]{Ism05}). For every $%
n\in \mathbb{N}$,%
\begin{equation}
\lbrack F_{n}^{t}(x)]^{\prime }=a(x,t;n)F_{n}^{t}(x)+b(x,t;n)F_{n-1}^{t}(x),
\label{[S2]-StructRelation}
\end{equation}%
where%
\begin{equation*}
\begin{array}{l}
a(x,t;n)=-4a_{n}^{2}(t)\,x, \\ 
b(x,t;n)=4a_{n}^{2}(t)\left( x^{2}-t+a_{n}^{2}(t)+a_{n+1}^{2}(t)\right) .%
\end{array}%
\end{equation*}

\item Holonomic equation (\cite[Th. 3.2.3, and p.57]{Ism05}). The SMOP $%
\{F_{n}^{t}\}_{n\geqslant 0}$\ satisfies the second order linear
differential equation 
\begin{equation}
\lbrack F_{n}^{t}(x)]^{\prime \prime }+R_{n}^{t}(x)[F_{n}^{t}(x)]^{\prime
}+S_{n}^{t}(x)F_{n}^{t}(x)=0,  \label{EcDiLi}
\end{equation}%
where%
\begin{eqnarray*}
R_{n}^{t}(x) &=&-4\big(x^{3}-tx\big)-\frac{2x}{%
x^{2}-t+a_{n}^{2}(t)+a_{n+1}^{2}(t)}, \\
S_{n}^{t}(x) &=&4a_{n}^{2}(t)\left[ 4x^{2}\left(
a_{n-1}^{2}(t)+a_{n}^{2}(t)+a_{n+1}^{2}(t)-t-{\frac{2}{%
x^{2}-t+a_{n}^{2}(t)+a_{n+1}^{2}(t)}}\right) \right. \\
&&\biggl.+\big(a_{n}^{2}(t)+a_{n+1}^{2}(t)-t\big)\big(%
a_{n-1}^{2}(t)+a_{n}^{2}(t)-t\big)+1\biggr].
\end{eqnarray*}
\end{enumerate}
\end{proposition}



The structure of the manuscript is as follows. Section \ref%
{[SECTION-2]-KernelsP} will be devoted to the Freud kernel polynomials. In
Section \ref{[SECTION-3]-ConnForm}, several connection formulas between
Freud-type and Freud polynomials are given and we briefly analyze the
coefficients in the three term recurrence relation for Freud-type
polynomials and a string-type equation for these coefficients. In Section %
\ref{[SECTION-4]-HEq-ElectrModel}, we provide the lowering and raising
operators associated with Freud-type polynomials, the corresponding
holonomic equation, and the electrostatic interpretation of the zeros of
Freud-type polynomials as positive unit charges interacting according to a
logarithmic potential under the action of an external field. Notice that the
second order linear differential equation is obtained as a composition of
the ladder operators following the ideas in \cite{CI-JPA97} for some
absolutely continuous measures and \cite{CG-JPA02} for measures with Dirac
deltas. Nevertheless, we must point out that in our case we consider a
measure with an extra time parameter $t$ and a Dirac mass located at $x=0$.
The explicit expressions of the polynomial coefficients in the holonomic
equation are given in our contribution. Its role in the study of the
electrostatic interpretation is emphasized in this Section. In Section \ref%
{[SECTION-5]-Zeros}, we analyze the zero behavior of Freud and Freud-type
polynomials in terms of the parameter $t$ as well as we deduce some results
concerning the monotonicity and speed of convergence of zeros of Freud-type
polynomials in terms of the mass $M$. Finally, in Section \ref%
{[SECTION-6]-NumExp}, we explore some numerical results showing the location
of certain zeros of Freud-type orthogonal polynomials.


\section{Freud kernel polynomials}

\label{[SECTION-2]-KernelsP}



The $n$-th degree kernel polynomial%
\begin{equation}
K_{n}(x,y;t)=\sum\limits_{k=0}^{n}\frac{F_{k}^{t}(x)F_{k}^{t}(y)}{%
||F_{k}^{t}||_{t}^{2}}  \label{Kernel-n}
\end{equation}%
associated with the polynomial sequence $\{F_{n}^{t}\}_{n\geqslant 0}$ will
play a key role in order to prove some of the basic results of the
manuscript. Notice that if $deg\ p\leqslant n$, then the $n$-th kernel
polynomial satisfies the so-called \textquotedblleft reproducing
property\textquotedblright 
\begin{equation}
\int_{\mathbb{R}}K_{n}\left( x,y;t\right) p\left( x\right) \omega
_{t}(x)dx=p\left( y\right) .  \label{Kn-PropRepro}
\end{equation}%
For $x\neq y$, according to the Christoffel-Darboux formula, for every $n\in 
\mathbb{N}$ we get%
\begin{equation}
K_{n}(x,y;t)=\frac{1}{||F_{n}^{t}||_{t}^{2}}\frac{%
F_{n+1}^{t}(x)F_{n}^{t}(y)-F_{n+1}^{t}(y)F_{n}^{t}(x)}{x-y}  \label{CristDar}
\end{equation}%
together with the confluent expression 
\begin{equation*}
K_{n}(x,x;t)=\sum_{k=0}^{n}\frac{[F_{k}^{t}(x)]^{2}}{||F_{k}^{t}||^{2}}=%
\frac{[F_{n+1}^{t}]^{\prime }(x)F_{n}^{t}(x)-[F_{n}^{t}]^{\prime
}(x)F_{n+1}^{t}(x)}{||F_{n}^{t}||^{2}}.
\end{equation*}%
From the above formula%
\begin{equation*}
K_{n}(0,0;t)=\frac{[F_{n+1}^{t}]^{\prime}(0)F_{n}^{t}(0)-[F_{n}^{t}]^{\prime
}(0)F_{n+1}^{t}(0)}{||F_{n}^{t}||_{t}^{2}},
\end{equation*}%
and also we conclude that $K_{n}(0,0;t)>0$, no matter the degree $n$ or
parity of the kernel, since it is a sum of strictly positive terms.

The weight function $\omega _{t}(x)$ is an even function, so the monic
polynomial sequence $\{F_{n}^{t}(x)\}_{n\geqslant 0}$ is symmetric (see \cite%
[Ch. I]{Chi78}) and the following expressions hold%
\begin{equation}
\left\{ 
\begin{array}{rcl}
F_{n}^{t}(0)=[F_{n}^{t}]^{\prime \prime }(0)=0, &  & \text{if }n\text{\ is
odd,} \\ 
&  &  \\ 
\lbrack F_{n}^{t}]^{\prime }(0)=[F_{n}^{t}]^{\prime \prime \prime }(0)=0, & 
& \text{if }n\text{\ is even.}%
\end{array}%
\right.  \label{F(0)Nulos}
\end{equation}%
From (\ref{Kernel-n}), (\ref{CristDar}) and (\ref{F(0)Nulos}) we have for $%
m\geq 1$%
\begin{equation}
\begin{array}{rll}
K_{2m-1}(x,0;t) & =\medskip & K_{2m-2}(x,0;t), \\ 
K_{2m}(x,0;t) & = & {\displaystyle\frac{1}{||F_{2m}^{t}||_{t}^{2}}\frac{%
F_{2m}^{t}(0)F_{2m+1}^{t}(x)}{x}}%
\end{array}
\label{KoddZERO}
\end{equation}%
and, therefore, all $K_{n}(x,0;t)$ has even degree for any $n$. It is
obvious that%
\begin{equation}
\begin{array}{rll}
K_{2m-1}(0,0;t) & =\medskip & K_{2m-2}(0,0;t), \\ 
K_{2m}(0,0;t) & = & {\displaystyle\frac{[F_{2m+1}^{t}]^{\prime
}(0)F_{2m}^{t}(0)}{||F_{2m}^{t}||_{t}^{2}}.}%
\end{array}
\label{[S2]-Bcn}
\end{equation}


From now on, $\{F{_{n}^{t,[k]}\}}_{n\geq 0}$ will denote the SMOP with
respect to the inner product%
\begin{equation}
\langle p,q\rangle _{t,[k]}=\int_{\mathbb{R}%
}p(x)q(x)x^{k}e^{-x^{4}+2tx^{2}}dx  \label{[S1]-InnProd-2}
\end{equation}%
which is a polynomial modification of the measure $d\mu
_{t}(x)=e^{-x^{4}+2tx^{2}}dx$ called the $k$\textit{-iterated Christoffel
perturbation}. If $k=1$ we have the \textit{Christoffel canonical
transformation of the measure }$\mu _{t}$\textit{\ }(see \cite{Z-JCAM97}, 
\cite{Y-BKMS02} among others). The SMOP $\{F{_{n}^{t,[k]}\}}_{n\geq 0}$ is
known as the \textit{monic }$k$\textit{-iterated Freud kernel polynomials}.
In this particular case, due to the parity of the weight function, when $k$
is odd the perturbed measure is not quasi-definite and, as consequence, the
corresponding SMOP does not exist. On the contrary, when $k$ is even the
perturbed measure is still positive definite, and therefore the
corresponding $k$-iterated Freud kernel polynomials are well defined. We
will denote by%
\begin{equation*}
||F_{n}^{t,[k]}||_{t,[k]}^{2}=\langle F_{n}^{t,[k]},F_{n}^{t,[k]}\rangle
_{t,[k]}=\int_{\mathbb{R}}|F_{n}^{t,[k]}(x)|^{2}x^{k}e^{-x^{4}+2tx^{2}}dx,%
\quad k\text{ even},
\end{equation*}%
the corresponding norm.






Let $\{F_{n}^{t,[2]}(x)\}_{n\geq 0}$ denote the sequence of $2$-iterated
monic Freud kernel polynomials, or the SMOP associated with the measure $%
d\mu _{t}^{[2]}=x^{2}d\mu _{t}$ which is the $2$\textit{-iterated
Christoffel transformation of }$\mu _{t}$.

The following two lemmas will be useful in the sequel.



\begin{lemma}
\label{CFKer[2]} The $2$-iterated Freud kernel polynomials and the Freud
orthogonal polynomials satisfy the connection formulas,%
\begin{eqnarray}
x^{2}F_{2m-1}^{t,[2]}(x) &=&F_{2m+1}^{t}(x)+\xi
_{2m-1}^{2}\,F_{2m-1}^{t}(x),\quad m\geq 1,  \label{Kernel[2]CFodd} \\
x^{2}F_{2m}^{t,[2]}(x) &=&F_{2m+2}^{t}(x)+\xi _{2m}^{2}\,F_{2m}^{t}(x),\quad
m\geq 1,  \label{Kernel[2]CFeven}
\end{eqnarray}%
where%
\begin{eqnarray*}
\xi _{2m-1}^{2} &=&\frac{||F_{2m-1}^{t,[2]}||_{t,[2]}^{2}}{%
||F_{2m-1}^{t}||_{t}^{2}}=\frac{-[F_{2m+1}^{t}]^{\prime }(0)}{%
[F_{2m-1}^{t}]^{\prime }(0)}, \\
\xi _{2m}^{2} &=&\frac{||F_{2m}^{t,[2]}||_{t,[2]}^{2}}{||F_{2m}^{t}||_{t}^{2}%
}=\frac{-F_{2m+2}^{t}(0)}{F_{2m}^{t}(0)}.
\end{eqnarray*}%
Furthermore, 
\begin{eqnarray}
x^{2}F_{2m-1}^{t,[2]}(x) &=&xF_{2m}^{t}(x)+\left[
\xi_{2m-1}^{2}-a_{2m}^{2}(t)\right] \,F_{2m-1}^{t}(x),  \label{Kernel[2]odd}
\\
xF_{2m}^{t,[2]}(x) &=&F_{2m+1}^{t}(x),  \label{Kernel[2]even}
\end{eqnarray}%
where%
\begin{equation}
\left[ \xi _{2m-1}^{2}-a_{2m}^{2}(t)\right] >0.  \label{coefPos}
\end{equation}
\end{lemma}



\begin{proof}
We can expand $x^{2}F_{n}^{t,[2]}(x)$ in terms of the SMOP $%
\{F_{n}^{t}\}_{n\geqslant 0}$ as%
\begin{equation*}
x^{2}F_{n}^{t,[2]}(x)=\sum_{k=0}^{n+2}b_{n,k}F_{k}^{t}(x).
\end{equation*}%
For every $k<n,$ we have $b_{n,k}=0$ by orthogonality. Also we have $%
b_{n,n+2}=1$, since we deal with monic polynomials. The following coefficient%
\begin{equation*}
b_{n,n+1}=\frac{\langle x^{2}F_{n}^{t,[2]},F_{n+1}^{t}\rangle _{t}}{%
||F_{n+1}^{t}||_{t}^{2}}
\end{equation*}%
vanishes, because the numerator is an integral in the whole $\mathbb{R}$\ of
an odd function, so%
\begin{equation*}
\langle x^{2}F_{n}^{t,[2]},F_{n+1}^{t}\rangle _{t}=\int_{-\infty }^{\infty
}x^{2}F_{n}^{t,[2]}F_{n+1}^{t}e^{-x^{4}+2tx^{2}}dx=0.
\end{equation*}%
Next, we easily see that%
\begin{equation*}
b_{n,n}=\xi _{n}^{2}=\frac{\langle x^{2}F_{n}^{t,[2]},F_{n}^{t}\rangle _{t}}{%
||F_{n}^{t}||_{t}^{2}}=\frac{\langle F_{n}^{t,[2]},F_{n}^{t}\rangle _{t,[2]}%
}{||F_{n}^{t}||_{t}^{2}}=\frac{||F_{n}^{t,[2]}||_{t,[2]}^{2}}{%
||F_{n}^{t}||_{t}^{2}}>0.
\end{equation*}%
Thus, we have%
\begin{equation}
x^{2}F_{n}^{t,[2]}(x)=F_{n+2}^{t}(x)+\xi _{n}^{2}\,F_{n}^{t}(x),
\label{Kernel[2]CF}
\end{equation}
Shifting the index in the above formula for $n$ odd and $n$ even, we obtain
formulas (\ref{Kernel[2]CFodd}) and (\ref{Kernel[2]CFeven}) respectively.
Taking $x=0$ in (\ref{Kernel[2]CFeven}) yields%
\begin{equation*}
\xi _{2m}^{2}=\frac{||F_{2m}^{t,[2]}||_{t,[2]}^{2}}{||F_{2m}^{t}||_{t}^{2}}=%
\frac{-F_{2m+2}^{t}(0)}{F_{2m}^{t}(0)},
\end{equation*}%
which is well defined since $x=0$ is not a zero of any Freud polynomial of
even degree. In case of Freud polynomials of odd degree, the above
expression is not defined. Letting $x\rightarrow 0$ in (\ref{Kernel[2]CFodd}%
), we may apply L'Hôpital's rule obtaining 
\begin{equation*}
\xi _{2m-1}^{2}=\frac{||F_{2m-1}^{t,[2]}||_{t,[2]}^{2}}{%
||F_{2m-1}^{t}||_{t}^{2}}=\frac{-[F_{2m+1}^{t}]^{\prime }(0)}{%
[F_{2m-1}^{t}]^{\prime }(0)}.
\end{equation*}
From (\ref{3TRR-Monic}) 
\begin{equation}
F_{2m+2}^{t}(x)=xF_{2m+1}^{t}(x)-a_{2m+1}^{2}(t)F_{2m}^{t}(x),
\label{RRTT2n2}
\end{equation}%
and taking $x=0$ we have%
\begin{equation*}
a_{2m+1}^{2}(t)=\frac{-F_{2m+2}^{t}(0)}{F_{2m}^{t}(0)},
\end{equation*}%
which essentially equals to $\xi _{2m}^{2}$. Combining (\ref{Kernel[2]CFeven}%
) with (\ref{RRTT2n2}) we have%
\begin{eqnarray*}
x^{2}F_{2m}^{t,[2]}(x) &=&xF_{2m+1}^{t}(x)+\left[ \xi
_{2m}^{2}-a_{2m+1}^{2}(t)\right] \,F_{2m}^{t}(x) \\
&=&xF_{2m+1}^{t}(x).
\end{eqnarray*}%
Dividing by $x$ we obtain the connection formula (\ref{Kernel[2]even}).

Next, from (\ref{3TRR-Monic}) 
\begin{equation*}
F_{2m+1}^{t}(x)=xF_{2m}^{t}(x)-a_{2m}^{2}(t)F_{2m-1}^{t}(x),
\end{equation*}%
and substituting this expression into (\ref{Kernel[2]CFodd}) gives 
\begin{equation*}
x^{2}F_{2m-1}^{t,[2]}(x)=xF_{2m}^{t}(x)+\left[ \xi _{2m-1}^{2}-a_{2m}^{2}(t)%
\right] \,F_{2m-1}^{t}(x)
\end{equation*}%
which yields expression (\ref{Kernel[2]odd}). The $x$ derivative of (\ref%
{Kernel[2]odd}) evaluated at $x=0$, provides 
\begin{equation}
\lbrack \xi _{2m-1}^{2}-a_{2m}^{2}(t)]=\frac{-F_{2m}^{t}(0)}{%
[F_{2m-1}^{t}]^{\prime }(0)},  \label{xi2odd-a2even}
\end{equation}

It is very well known (see \cite[Ch. I, §8]{Chi78}) that, due to the
symmetry of the weight function, there exist two $n$-th degree monic
polynomials $A_{m}$ and $B_{m}$ such that 
\begin{eqnarray*}
F_{2m}^{t}(x) &=&A_{m}(x^{2}), \\
F_{2m+1}^{t}(x) &=&xB_{m}(x^{2}),
\end{eqnarray*}%
being $B_{m}(x)$ the kernel polynomials, with parameter $0$, of $A_{n}(x)$.
Thus, we have that $sign\, F_{2m}^{t}(0)=sign\,A_{m}(0)=(-1)^{m}$ and $%
sign\,[F_{2m-1}^{t}]^{\prime}(0)=sign\,B_{m-1}(0)=(-1)^{m-1}$, so therefore,
concerning the sign of (\ref{xi2odd-a2even}) we have 
\begin{equation*}
sign\frac{F_{2m}^{t}(0)}{[F_{2m-1}^{t}]^{\prime }(0)}=sign\frac{A_{m}(0)}{%
B_{m}(0)}=\frac{(-1)^{m}}{(-1)^{m-1}}=-1,
\end{equation*}%
which shows that%
\begin{equation*}
[ \xi _{2m-1}^{2}-a_{2m}^{2}(t)]>0.
\end{equation*}%
This completes the proof.
\end{proof}


\begin{lemma}
\label{3TRRKer[2]} The SMOP $\{F_{n}^{t,[2]}\}_{n\geqslant 0}$ satisfies the
following three term recurrence relation 
\begin{equation*}
xF_{n}^{t,[2]}(x)=F_{n+1}^{t,[2]}(x)+\alpha _{n}^{2}F_{n-1}^{t,[2]}(x),
\end{equation*}
where%
\begin{equation*}
\alpha _{n}^{2}=\frac{||F_{n}^{t,[2]}||_{t,[2]}^{2}}{%
||F_{n-1}^{t,[2]}||_{t,[2]}^{2}}=\frac{\xi _{n}^{2}}{\xi _{n-1}^{2}}%
a_{n}^{2}(t).
\end{equation*}%
$a_{n}^{2}(t)$ and $\xi _{n}^{2}$ are given in (\ref{3TRR-Monic}) and Lemma %
\ref{CFKer[2]}, respectively.
\end{lemma}



\begin{proof}
We can expand $xF_{n}^{t,[2]}(x)$ in terms of $\{F_{n}^{t,[2]}\}_{n\geqslant
0}$ and, for orthogonality reasons, the only terms remaining are $%
F_{n+1}^{t,[2]}(x)$ and $F_{n-1}^{t,[2]}(x)$. The coefficient of $%
F_{n+1}^{t,[2]}(x)$ is $1$\ because we deal with monic polynomials and the
other one is%
\begin{equation*}
\alpha _{n}^{2}=\frac{||F_{n}^{t,[2]}||_{t,[2]}^{2}}{%
||F_{n-1}^{t,[2]}||_{t,[2]}^{2}}>0.
\end{equation*}%
Let notice that%
\begin{equation*}
\alpha _{n}^{2}=\frac{\langle F_{n}^{t,[2]},F_{n}^{t}\rangle _{t,[2]}}{%
\langle F_{n-1}^{t,[2]},F_{n-1}^{t}\rangle _{t,[2]}}=\frac{\langle
x^{2}F_{n}^{t,[2]},F_{n}^{t}\rangle _{t}}{\langle
x^{2}F_{n-1}^{t,[2]},F_{n-1}^{t}\rangle _{t}}.
\end{equation*}%
Applying (\ref{Kernel[2]CF}) we get%
\begin{eqnarray*}
\alpha _{n}^{2} &=&\frac{\langle F_{n+2}^{t},F_{n}^{t}\rangle _{t}+\xi
_{n}^{2}\langle F_{n}^{t},F_{n}^{t}\rangle _{t}}{\langle
F_{n+1}^{t},F_{n-1}^{t}\rangle _{t}+\xi _{n-1}^{2}\langle
F_{n-1}^{t},F_{n-1}^{t}\rangle _{t}} \\
&=&\frac{\xi _{n}^{2}}{\xi _{n-1}^{2}}\frac{||F_{n}^{t}||_{t}^{2}}{%
||F_{n-1}^{t}||_{t}^{2}} \\
&=&\frac{\xi _{n}^{2}}{\xi _{n-1}^{2}}a_{n}^{2}(t).
\end{eqnarray*}%
This completes the proof.
\end{proof}


Let $x_{n,k}=x_{n,k}(t)$, $x_{n,k}^{[2]}=x_{n,k}^{[2]}(t)$, $k=1,\ldots ,n$,
be the zeros of $F_{n}^{t}(x)$, $F_{n}^{t,[2]}(x),$ respectively, arranged
in an increasing order. All of them are real and simple. By parity reasons,
these zeros are symmetrically arranged with respect to the origin. That is, $%
-x_{n,1}^{[2]}=x_{n,n}^{[2]}$, $-x_{n,2}^{[2]}=x_{n,n-1}^{[2]}$ and so on.

We next prove that the zeros of $G_{2m}(x)=xF_{2m-1}^{t,[2]}(x)$ and $%
F_{2m}^{t}(x)$ interlace. Concerning the zeros $g_{2m,k}=g_{2m,k}(t)$, $%
k=1,\ldots ,2m,$\ of $G_{2m}(x)$ are the same zeros of $F_{2m-1}^{t,[2]}(x)$
except one more zero at the origin, so $G_{2m}(x)$ has a double zero at $x=0$%
. Thus, we have $g_{2m,l}=x_{2m-1,l}^{[2]}$, $l=1,\ldots m$, with $%
x_{2m-1,m}^{[2]}=g_{2m,m}=g_{2m,m+1}=0$ and $g_{2m,r}=x_{2m-1,r-1}^{[2]}$, $%
r=m+1,\ldots ,2m$.

Both $G_{2m}(x)$ and $F_{2m}^{t}(x)$ are even polynomial functions, so their
respective graphs are symmetric with respect to the point $x=0$, which means
that we only need to prove interlacing in the positive real semi-axis, being
the situation in $\mathbb{R}_{-}$ the reflection with respect to the $y$%
-axis. Hence, we only need to prove that, for $x>0$, between two consecutive
zeros $(x_{2m,k}\,,x_{2m,k+1})$ there are only one zero of\ $%
F_{2m-1}^{t,[2]}(x)$. Consider (\ref{Kernel[2]CFodd}) evaluated at the
positive zeros of $F_{2m}^{t}(x)$ (i.e., $x_{2m,r}>0$, with $r=m+1,\ldots
,2m $). We have 
\begin{equation*}
x_{2m,r}\,F_{2m-1}^{t,[2]}(x_{2m,r})=\left[ \xi _{2m-1}^{2}-a_{2m}^{2}(t)%
\right] \,F_{2m-1}^{t}(x_{2m,r}),
\end{equation*}
and taking into account (\ref{coefPos}), we get 
\begin{equation}
sign[F_{2m-1}^{t,[2]}(x_{2m,r})]=sign[F_{2m-1}^{t}(x_{2m,r})],\ r=m+1,\ldots
,2m.  \label{Z2}
\end{equation}
Thus, from (\ref{Z2}), the relation between the zeros of $G_{2m}(x)$ and $%
F_{2m-1}^{t,[2]}(x)$, the symmetric reflection with respect to the $y$-axis,
and the well known fact that the zeros of $F_{2m-1}^{t}(x)$ interlace with
the zeros of $F_{2m}^{t}(x)$, we obtain the following interlacing property



\begin{theorem}
\label{T1} The inequalities%
\begin{equation*}
x_{2m,1}<g_{2m,1}<x_{2m,2}<\cdots
<x_{2m,m}<g_{2m,m}=0=g_{2m,m+1}<\cdots<x_{2m,2m-1}<g_{2m,2m}<x_{2m,2m}\,,
\end{equation*}%
hold for every $m\in \mathbb{N}$.
\end{theorem}



\section{Connection formulas}

\label{[SECTION-3]-ConnForm}



We can expand $Q_{n}^{t}(x)$\ in terms of the SMOP $\{F_{n}^{t}\}_{n%
\geqslant 0}$, which is an orthogonal basis of $\mathbb{P}$, as follows%
\begin{equation*}
Q_{n}^{t}(x)=F_{n}^{t}\left( x\right) +\sum_{i=0}^{n-1}\lambda
_{n,i}F_{i}^{t}(x),  \label{expan-11}
\end{equation*}%
where%
\begin{equation*}
\lambda _{n,i}=\frac{\left\langle F_{i}^{t}\left( x\right)
,Q_{n}^{t}(x)\right\rangle _{t}}{||F_{i}^{t}||_{t}^{2}},\quad 0\leqslant
i\leqslant n-1.
\end{equation*}
From (\ref{Freud-typeInnProd}) and (\ref{Kernel-n}) the above equality
becomes%
\begin{equation}
Q_{n}^{t}(x)=F_{n}^{t}(x)-MQ_{n}^{t}(0)K_{n-1}(x,0;t),  \label{CF-1}
\end{equation}%
and evaluating the above expression at $x=0$, we deduce%
\begin{equation}
Q_{n}^{t}(0)=\frac{F_{n}^{t}(0)}{1+MK_{n-1}(0,0;t)}.  \label{CF-Q0}
\end{equation}%
Hence, from (\ref{3TRR-Monic}) and (\ref{CristDar}) , (\ref{CF-1}) reads as%
\begin{equation*}
xQ_{n}^{t}(x)=F_{n+1}^{t}(x)+A_n^t F_n^t(x)+B_{n}^{t}F_{n-1}^{t}(x),
\end{equation*}%
where%
\begin{eqnarray*}
A_{n}^{t} &=&-\frac{MF_{n}^{t}(0)F_{n-1}^{t}(0)}{\left\Vert
F_{n-1}^{t}\right\Vert _{t}^{2}\big(1+MK_{n-1}(0,0;t)\big)}, \\
B_{n}^{t} &=&a_{n}^{2}(t)+\frac{M\left( F_{n}^{t}(0)\right) ^{2}}{\left\Vert
F_{n-1}^{t}\right\Vert _{t}^{2}\big(1+MK_{n-1}(0,0;t)\big)}%
=a_{n}^{2}(t)\left( \frac{1+MK_{n}(0,0;t)}{1+MK_{n-1}(0,0;t)}\right) .
\end{eqnarray*}

\begin{remark}
Due to the fact that $\omega _{t}(x)$ is an even weight function, $x=0$ is
always a zero of $F_{n}^{t}(x)$\ for $n$ odd. Then, $%
F_{n}^{t}(0)F_{n-1}^{t}(0)=0$ for every $n\geqslant 1$, and therefore $%
A_{n}^{t}=0 $ for all positive integer $n$.
\end{remark}

From the above remark, we have%
\begin{align}
xQ_{n}^{t}(x)& =F_{n+1}^{t}(x)+B_{n}^{t}F_{n-1}^{t}(x),  \label{ConnForm_1}
\\
Q_{2n+1}^{t}(x)& =F_{2n+1}^{t}(x).  \notag
\end{align}%
Introducing the notation%
\begin{equation}
b_{n}^{t}=\dfrac{1+MK_{n}(0,0;t)}{1+MK_{n-1}(0,0;t)}  \label{ratios-1}
\end{equation}%
we get%
\begin{equation*}
B_{n}^{t}=a_{n}^{2}(t)\,b_{n}^{t}.
\end{equation*}

This yields an expression for the ratio of the energy of polynomials $%
Q_{n}^{t}(x)$ and $F_{n}^{t}(x)$ with respect to the norms associated with
their corresponding inner products.

\begin{proposition}
\label{FEnergy}Let $\left\Vert \cdot \right\Vert ^{2}$ be the squared norm
of Freud-type monic polynomials with respect to (\ref{Freud-typeInnProd}).
Then%
\begin{equation*}
\frac{||Q_{n}^{t}||^{2}}{||F_{n}^{t}||_{t}^{2}}=\frac{1+MK_{n}\left(
0,0;t\right) }{1+MK_{n-1}\left( 0,0;t\right) }=b_{n}^{t},\quad n\geqslant 1.
\end{equation*}%
Moreover, $b_{n}^{t}=1$ when $n$ is odd and $b_{n}^{t}>1$ when $n$ is even,
i.e., for every $m\geqslant 0$, 
\begin{equation*}
||Q_{2m+1}^{t}||^{2}=||F_{2m+1}^{t}||_{t}^{2}.
\end{equation*}
\end{proposition}

\begin{proof}
Taking in account%
\begin{eqnarray*}
\left\Vert Q_{n}^{t}\right\Vert ^{2} &=&\left\langle
Q_{n}^{t}(x),x^{n}\right\rangle \\
&=&\left\langle Q_{n}^{t}(x),F_{n}^{t}\left( x\right) \right\rangle \\
&=&\left\langle Q_{n}^{t}(x),F_{n}^{t}\left( x\right) \right\rangle
_{t}+MQ_{n}^{t}(0)F_{n}^{t}\left( 0\right)
\end{eqnarray*}%
and using (\ref{CF-Q0}), we get%
\begin{eqnarray*}  \label{normQ}
||Q_{n}^{t}||^{2} &=&||F_{n}^{t}||_{t}^{2}+\dfrac{M\bigl(F_{n}^{t}\left(
0\right) \bigr)^{2}}{1+MK_{n-1}\left( 0,0;t\right) }  \notag \\
&=&||F_{n}^{t}||_{t}^{2}\,\frac{\bigl(1+MK_{n-1}\left( 0,0;t\right) \bigr)+%
\frac{M\left( F_{n}^{t}\left( 0\right) \right) ^{2}}{\left\Vert
F_{n}^{t}\right\Vert _{t}^{2}}}{1+MK_{n-1}\left( 0,0;t\right) }  \notag \\
&=&||F_{n}^{t}||_{t}^{2}\,\frac{1+MK_{n}\left( 0,0;t\right) }{%
1+MK_{n-1}\left( 0,0;t\right) }.
\end{eqnarray*}

Evaluating (\ref{KoddZERO}) at $x=0$ we have $b_{n}^{t}=1$ for $n$ odd and $%
b_{n}^{t}>1$ for $n$ even. This gives the result when combined with the
above equation.
\end{proof}



As a consequence, we get



\begin{theorem}
\label{xQ=F} Let $\left\{ Q_{n}^{t}\right\} _{n\geqslant 0}$ be the sequence
of monic Freud-type polynomials orthogonal with respect to (\ref%
{Freud-typeInnProd}). Then%
\begin{equation*}
Q_{2m+1}^{t}(x)=F_{2m+1}^{t}(x),\quad m\geqslant 0,
\end{equation*}%
\begin{equation*}
Q_{2m}^{t}(x)=F_{2m}^{t}(x)-\frac{MF_{2m}^{t}(0)}{1+MK_{2m-1}(0,0;t)}%
K_{2m-1}(x,0;t),\quad m\geqslant 1.
\end{equation*}
\end{theorem}



Next, we provide an alternative way to represent the Freud-type polynomials
of even degree $Q_{2m}^{t}(x)$ in terms of the polynomials $F_{2m}^{t}(x)$
and the $2$-iterated monic Freud kernel polynomials $F_{n}^{t,[2]}(x)$. This
representation will allow us to obtain the results about monotonicity and
asymptotic behavior (presented below in this work) for the zeros of $%
Q_{2m}^{t}(x)$ in terms of the parameter $M$ present in (\ref%
{Freud-typeInnProd}). We only need to consider Freud type polynomials of
even degree, because they are the only ones affected by variations of $M$.



\begin{theorem}[Connection formula]
\label{[S1]-THEO-1}The sequence $\{{\tilde{Q}_{2m}^{t}\}}_{m\geqslant 0}$\
can be represented as 
\begin{equation}
\tilde{Q}_{2m}^{t}(x)=F_{2m}^{t}(x)+MK_{2m-1}(0,0;t)G_{2m}(x),\quad
m\geqslant 1,  \label{[S2]-ConnForm-Main}
\end{equation}%
with $G_{2m}(x)=xF_{2m-1}^{t,[2]}(x)$, ${\tilde{Q}_{2m}^{t}(x)=\kappa }_{2m}{%
Q_{n}^{t}(x)}$ and ${\kappa}_{2m}=1+MK_{2m-1}(0,0;t)>0$.
\end{theorem}


\begin{proof}
Let ${\kappa }_{2m}$\ be given by the positive quantity ${\kappa }%
_{2m}=1+MK_{2m-1}(0,0;t)$. From the expression for $Q_{2m}^{t}(x)$\ in
Theorem \ref{xQ=F} we have%
\begin{equation*}
{\kappa }_{2m}Q_{2m}^{t}(x)={\kappa }%
_{2m}F_{2m}^{t}(x)-MF_{2m}^{t}(0)K_{2m-1}(x,0;t),\quad m\geqslant 1.
\end{equation*}

Being ${\kappa }_{2m}>0$, we call $\tilde{Q}_{2m}^{t}(x)={\kappa }%
_{2m}Q_{2m}^{t}(x)$ the polynomial with the same zeros that $Q_{2m}^{t}(x)$.
Next, (\ref{KoddZERO}) and (\ref{[S2]-Bcn}) yields%
\begin{eqnarray}
\tilde{Q}_{2m}^{t}(x) &=&F_{2m}^{t}(x)+M\left[
K_{2m-1}(0,0;t)F_{2m}^{t}(x)-F_{2m}^{t}(0)K_{2m-1}(x,0;t)\right]  \notag \\
&=&F_{2m}^{t}(x)-M\frac{[F_{2m-1}^{t}]^{\prime }(0)F_{2m}^{t}(0)}{%
||F_{2m-1}^{t}||_{t}^{2}}\left[ F_{2m}^{t}(x)-\frac{F_{2m}^{t}(0)}{%
[F_{2m-1}^{t}]^{\prime }(0)}\frac{F_{2m-1}^{t}(x)}{x}\right]  \label{squaQ2m}
\\
&=&F_{2m}^{t}(x)+MK_{2m-1}(0,0;t)\left[ F_{2m}^{t}(x)-\frac{F_{2m}^{t}(0)}{%
[F_{2m-1}^{t}]^{\prime }(0)}\frac{F_{2m-1}^{t}(x)}{x}\right]  \notag
\end{eqnarray}

On the other hand, replacing (\ref{xi2odd-a2even}) into (\ref{Kernel[2]odd})
and dividing by $x$ yields 
\begin{equation*}
xF_{2m-1}^{t,[2]}(x)=F_{2m}^{t}(x)-\frac{F_{2m}^{t}(0)}{[F_{2m-1}^{t}]^{%
\prime }(0)}\frac{F_{2m-1}^{t}(x)}{x}.
\end{equation*}%
From the above expression, we can finally rewrite (\ref{squaQ2m}) as%
\begin{equation*}
\tilde{Q}_{2m}^{t}(x)=F_{2m}^{t}(x)+MK_{2m-1}(0,0;t)xF_{2m-1}^{t,[2]}(x).
\end{equation*}%
This completes the proof.
\end{proof}



In the remaining of this section we will focus our attention on the
coefficients of the three term recurrence relation satisfied by the
Freud-type SMOP. Since $Q_{n}^{t}(x)$ are standard and symmetric, they
satisfy the following fundamental recurrence relation.



\begin{proposition}
\label{RR3TQ} The polynomials $Q_{n}^{t}(x)$ satisfy the three term
recurrence relation%
\begin{equation}
xQ_{n}^{t}(x)=Q_{n+1}^{t}(x)+\gamma _{n}(t)Q_{n-1}^{t}(x),
\label{RR3T-Tilde}
\end{equation}%
where%
\begin{equation}
\gamma _{n}(t)=\frac{b_{n}^{t}}{b_{n-1}^{t}}a_{n}^{2}(t).
\label{gamma-tilde}
\end{equation}
\end{proposition}



\begin{proof}
We expand $xQ_{n}^{t}(x)$ in terms of the SMOP $\{Q_{n}^{t}\}_{n\geqslant 0}$%
and, taking into account (\ref{Freud-typeInnProd}), the Lemma follows.
\end{proof}



\begin{proposition}[String equation]
The coefficients (\ref{gamma-tilde}) of the above three term recurrence
relation for $\left\{ Q_{n}^{t}\right\} _{n\geqslant 0}$\ satisfy the
following nonlinear difference string equation%
\begin{equation*}  \label{StringPertQ}
4\gamma _{n}^{2}(t)\left( \frac{b_{n-2}^{t}}{b_{n}^{t}}\gamma
_{n-1}^{2}(t)+\left( \frac{b_{n-1}^{t}}{b_{n}^{t}}\right) ^{2}\gamma
_{n}^{2}(t)+\frac{b_{n-1}^{t}}{b_{n+1}^{t}}\gamma _{n+1}^{2}(t)-\frac{t}{2}%
\right) =n.
\end{equation*}
\end{proposition}


\begin{proof}
It is enough to replace (\ref{gamma-tilde}) in the string equation (\ref%
{StringEq0}), and then the Proposition follows.
\end{proof}



\section{Holonomic equation and electrostatic model}

\label{[SECTION-4]-HEq-ElectrModel}



Next, we give details of the second order linear differential equation
satisfied by $\{Q_{n}^{t}\}_{n\geqslant 0}$ when $t>0$. First, from (\ref%
{3TRR-Monic}) we can rewrite (\ref{ConnForm_1}) as%
\begin{equation}
xQ_{n}^{t}(x)=A_{1}(x,t;n)F_{n}^{t}(x)+B_{1}(t;n)F_{n-1}^{t}(x),
\label{[S3]-hxConnForm}
\end{equation}%
with%
\begin{eqnarray*}
A_{1}(x,t;n) &=&x, \\
B_{1}(t;n) &=&a_{n}^{2}(t)\left( b_{n}^{t}-1\right) .
\end{eqnarray*}

In order to obtain the ladder operators and the second order linear
differential equation, we follow a different approach as in \cite[Ch. 3]%
{Ism05}. Our technique is based on the connection formula (\ref{CF-1}), the
three term recurrence relation (\ref{3TRR-Monic}) satisfied by the SMOP $%
\{F_{n}^{t}\}_{n\geqslant 0}$, and its corresponding structure relation (\ref%
{[S2]-StructRelation}).

We begin by proving several lemmas which are needed for the proof of Theorem %
\ref{[S2]-THEO-4}.


\begin{lemma}
\label{[S3]-LEMMA-2}For the SMOP $\{Q_{n}^{t}\}_{n\geqslant 0}$ and $%
\{F_{n}^{t}\}_{n\geqslant 0}$ we have%
\begin{equation}
x[Q_{n}^{t}{(x)}]^{\prime
}=C_{1}(x,t;n)F_{n}^{t}(x)+D_{1}(x,t;n)F_{n-1}^{t}(x),
\label{[S3]-DxSn-C1D1}
\end{equation}%
where%
\begin{equation*}
\begin{array}{ll}
C_{1}(x,t;n)= & -4a_{n}^{2}(t)\left[ b_{n}^{t}\,x^{2}+\left(
b_{n}^{t}-1\right) \left( a_{n-1}^{2}(t)+a_{n}^{2}(t)-t\right) \right]
,\medskip \\ 
D_{1}(x,t;n)= & 4{a_{n}^{2}(t)}\,x\left( {a_{n+1}^{2}(t)+b_{n}^{t}}\left[
x^{2}-t+a_{n}^{2}(t)\right] \right) -\frac{1}{x}{a_{n}^{2}(t)\big(b_{n}^{t}-1%
\big)}.%
\end{array}%
\end{equation*}%
The coefficients $A_{1}(x,t;n)$, $B_{1}(t;n)$ are given in (\ref%
{[S3]-hxConnForm}), $b_{n}^{t}$ is given in (\ref{ratios-1}), and $a(x,t;n)$%
, $b(x,t;n)$ come from the structure relation (\ref{[S2]-StructRelation}).
\end{lemma}



\begin{proof}
Shifting the index in (\ref{[S2]-StructRelation}) as $n\rightarrow n-1$, and
using (\ref{3TRR-Monic}) we obtain%
\begin{equation}
\lbrack F_{n-1}^{t}(x)]^{\prime }=\tilde{a}(x,t;n)F_{n}^{t}(x)+\tilde{b}%
(x,t;n)F_{n-1}^{t}(x),  \label{[S3]-Pnm1Der}
\end{equation}%
where%
\begin{eqnarray*}
\tilde{a}(x,t;n) &=&-4\left[ x^{2}-t+a_{n-1}^{2}(t)+a_{n}^{2}(t)\right], \\
\tilde{b}(x,t;n) &=&4x\left[ x^{2}-t+a_{n}^{2}(t)\right] .
\end{eqnarray*}%
Next, taking derivatives with respect to the variable $x$ in both sides of (%
\ref{[S3]-hxConnForm}), we get%
\begin{equation*}
x^{2}[Q^{t}{_{n}(x)]}^{\prime }=x^{2}[F_{n}^{t}(x)]^{\prime
}+xB_{1}(t;n)[F_{n-1}^{t}(x)]^{\prime }-B_{1}(t;n)F_{n-1}^{t}(x).
\end{equation*}%
Substituting (\ref{[S2]-StructRelation}), (\ref{[S3]-Pnm1Der}) and the
expression for $A_{1}(x,t;n)$, $B_{1}(t;n)$ into the above expression the
Lemma follows.
\end{proof}



From Proposition \ref{FEnergy}, $b_{n}^{t}=1$ when $n$ is odd, so the above
results can be simplified as%
\begin{equation*}
x[Q_{2m+1}^{t}{(x)}]^{\prime
}=C_{1}(x,t;2m+1)F_{2m+1}^{t}(x)+D_{1}(x,t;2m+1)F_{2m}^{t}(x),
\end{equation*}%
where%
\begin{equation*}
\begin{array}{ll}
C_{1}(x,t;2m+1)= & -4a_{2m+1}^{2}(t)\,x^{2},\medskip \\ 
D_{1}(x,t;2m+1)= & 4{a_{2m+1}^{2}(t)}\,x\left(
x^{2}-t+a_{2(n+1)}^{2}(t)+a_{2m+1}^{2}(t)\right) ,%
\end{array}%
\end{equation*}%
which corresponds to \eqref{[S2]-StructRelation}.



\begin{lemma}
\label{[S3]-LEMMA-3} The sequences of monic polynomials $\{Q_{n}^{t}\}_{n%
\geqslant 0}$ and $\{F_{n}^{t}\}_{n\geqslant 0}$ are also related by%
\begin{eqnarray}
xQ^{t}{_{n-1}(x)} &=&A_{2}(t;n)F_{n}^{t}(x)+B_{2}(x,t;n)F_{n-1}^{t}(x),
\label{[S3]-Snm1-A2D2} \\
x[Q_{n-1}^{t}(x)]^{\prime }
&=&C_{2}(x,t;n)F_{n}^{t}(x)+D_{2}(x,t;n)F_{n-1}^{t}(x),
\label{[S3]-DxSnm1-C2D2}
\end{eqnarray}%
where%
\begin{equation*}
\begin{split}
A_{2}(t;n)& =1-b_{n-1}^{t}, \\
B_{2}(x,t;n)& =xb_{n-1}^{t},
\end{split}%
\end{equation*}%
and%
\begin{equation*}
\begin{array}{ll}
C_{2}(x,t;n)= & (b_{n-1}^{t}-1)\left( 4xa_{n}^{2}(t)+\frac{1}{x}\right)
-4xb_{n-1}^{t}\left[ x^{2}-t+a_{n-1}^{2}(t)+a_{n}^{2}(t)\right] ,\medskip \\ 
D_{2}(x,t;n)= & 4a_{n}^{2}(t)\left( 1-b_{n-1}^{t}\right) \left[
x^{2}-t+a_{n}^{2}(t)+a_{n+1}^{2}(t)\right] +4x^{2}b_{n-1}^{t}\left[
x^{2}-t+a_{n}^{2}(t)\right] .%
\end{array}%
\end{equation*}
\end{lemma}



\begin{proof}
The proof of (\ref{[S3]-Snm1-A2D2}) and (\ref{[S3]-DxSnm1-C2D2}) is a
straightforward consequence of \eqref{3TRR-Monic}, (\ref{[S2]-StructRelation}%
), \eqref{[S3]-hxConnForm}, and Lemma \ref{[S3]-LEMMA-2}.
\end{proof}



By Proposition \ref{FEnergy} it is obvious that \eqref{[S3]-Snm1-A2D2} is
exactly the first equation of Theorem \ref{xQ=F} if $n$ is even and %
\eqref{[S3]-DxSnm1-C2D2} can be simplified as 
\begin{equation}
x[Q_{2m-1}^{t}{(x)}]^{\prime
}=C_{2}(x,t;2m)F_{2m}^{t}(x)+D_{2}(x,t;2m)F_{2m-1}^{t}(x),
\label{[S3]-DxSn-C2D2-even}
\end{equation}%
where%
\begin{equation*}
\begin{array}{ll}
C_{2}(x,t;2m)= & -4x\left( x^{2}-t+a_{2m-1}^{2}(t)+a_{2m}^{2}(t)\right) , \\ 
&  \\ 
D_{2}(x,t;2m)= & 4x^{2}\left( x^{2}-t+a_{2m}^{2}(t)\right) .%
\end{array}%
\end{equation*}%
An equivalent formulation of \eqref{[S3]-DxSn-C2D2-even} is (\ref%
{[S2]-StructRelation}) if we substitute $F_{2m}^{t}(x)$ in the above
relation according to \eqref{3TRR-Monic} when the index $n$ is shifted by $%
2m-1$.



The following lemma shows the converse of (\ref{[S3]-hxConnForm})--(\ref%
{[S3]-Snm1-A2D2}) for the polynomials $F_{n}^{t}(x)$ and $F_{n-1}^{t}(x)$.
Indeed, we express these two consecutive polynomials of the SMOP $%
\{F_{n}^{t}\}_{n\geqslant 0}$ in terms of only two consecutive Freud-type
orthogonal polynomials of the SMOP $\{Q_{n}^{t}\}_{n\geqslant 0}$.



\begin{lemma}
\label{[S3]-LEMMA-4}For $t>0$,%
\begin{eqnarray}
xb_{n-1}^{t} F_{n}^{t}(x) &=&xb_{n-1}^{t}Q_{n}^{t}(x)-a_{n}^{2}(t)\left(
b_{n}^{t}-1\right) Q_{n-1}^{t}(x),  \label{[S3]-InvR-Pn} \\
&&  \notag \\
xb_{n-1}^{t}F_{n-1}^{t}(x) &=&\left( b_{n-1}^{t}-1\right)
Q_{n}^{t}(x)+xQ_{n-1}^{t}(x).  \label{[S3]-InvR-Pnm1}
\end{eqnarray}
\end{lemma}



\begin{proof}
Note that (\ref{[S3]-hxConnForm}) and (\ref{[S3]-Snm1-A2D2}) can be
interpreted as a system of two linear equations with two polynomial
unknowns, namely $F_{n}^{t}(x)$ and $F_{n-1}^{t}(x)$. Hence from Cramer's
rule, we have%
\begin{equation*}
\Delta (x,t;n)=x^{2}b_{n-1}^{t}+a_{n}^{2}(t)\left( b_{n}^{t}-1\right) \left(
b_{n-1}^{t}-1\right) .
\end{equation*}%
From Proposition \ref{FEnergy} we know that $b_{n}^{t}=1$ when $n$ is odd
and $b_{n}^{t}>1$ when $n$ is even. Therefore, the product $\left(
b_{n}^{t}-1\right) \left( b_{n-1}^{t}-1\right) $ is always zero, and the
Lemma easily follows.
\end{proof}



\begin{theorem}[Ladder operators]
\label{[S2]-THEO-4}Let $\mathfrak{a}_{n}$\ and $\mathfrak{a}_{n}^{\dag }$\
be the differential operators%
\begin{eqnarray*}
\mathfrak{a}_{n} &=&%
\begin{vmatrix}
A_{2}(t;n) & C_{1}(x,t;n) \\ 
B_{2}(x,t;n) & D_{1}(x,t;n)%
\end{vmatrix}%
+\Delta (x,t;n)\frac{d}{dx}, \\
\mathfrak{a}_{n}^{\dag } &=&%
\begin{vmatrix}
A_{1}(x,t;n) & C_{2}(x,t;n) \\ 
B_{1}(t;n) & D_{2}(x,t;n)%
\end{vmatrix}%
-\Delta (x,t;n)\frac{d}{dx},
\end{eqnarray*}%
satisfying%
\begin{eqnarray}
\mathfrak{a}_{n}[Q_{n}^{t}(x)] &=&%
\begin{vmatrix}
A_{1}(x,t;n) & C_{1}(x,t;n) \\ 
B_{1}(t;n) & D_{1}(x,t;n)%
\end{vmatrix}%
Q_{n-1}^{t}(x),  \label{[S2]-LoweringEq} \\
\mathfrak{a}_{n}^{\dag }[Q_{n-1}^{t}(x)] &=&%
\begin{vmatrix}
A_{2}(t;n) & C_{2}(x,t;n) \\ 
B_{2}(x,t;n) & D_{2}(x,t;n)%
\end{vmatrix}%
Q_{n}^{t}(x).  \label{[S2]-RaisingEq}
\end{eqnarray}


Let point out that all the above expressions are given only in terms of the
coefficients in \eqref{[S3]-hxConnForm}, \eqref{3TRR-Monic}, %
\eqref{[S3]-DxSn-C1D1}, \eqref{[S3]-Snm1-A2D2}, and \eqref{[S3]-DxSnm1-C2D2}%
. Thus, $\mathfrak{a}_{n}$\ and $\mathfrak{a}_{n}^{\dag }$\ are,\
respectively, lowering and raising operators associated with the Freud-type
SMOP $\{Q_{n}^{t}\}_{n\geqslant 0}$.
\end{theorem}



\begin{proof}
The proof of Theorem \ref{[S2]-THEO-4} follows from Lemmas \ref{[S3]-LEMMA-2}%
--\ref{[S3]-LEMMA-4}. Replacing (\ref{[S3]-InvR-Pn})--(\ref{[S3]-InvR-Pnm1})
in (\ref{[S3]-DxSn-C1D1}) and (\ref{[S3]-DxSnm1-C2D2}) one obtains the
ladder equations%
\begin{eqnarray*}
\lbrack Q_{n}^{t}(x)]^{\prime } &=&\frac{%
C_{1}(x,t;n)B_{2}(x,t;n)-A_{2}(t;n)D_{1}(x,t;n)}{\Delta (x,t;n)}Q_{n}^{t}(x)
\\
&&\quad \quad \quad \quad \quad \quad \quad \quad +\frac{%
A_{1}(x,t;n)D_{1}(x,t;n)-B_{1}(t;n)C_{1}(x,t;n)}{\Delta (x,t;n)}%
Q_{n-1}^{t}(x)
\end{eqnarray*}%
and%
\begin{eqnarray*}
\lbrack Q_{n-1}^{t}(x)]^{\prime } &=&\frac{%
C_{2}(x,t;n)B_{2}(x,t;n)-A_{2}(t;n)D_{2}(x,t;n)}{\Delta (x,t;n)}Q_{n}^{t}(x)
\\
&&\quad \quad \quad \quad \quad \quad \quad \quad +\frac{%
A_{1}(x,t;n)D_{2}(x,t;n)-B_{1}(t;n)C_{2}(x,t;n)}{\Delta (x,t;n)}%
Q_{n-1}^{t}(x),
\end{eqnarray*}%
which are equivalent to (\ref{[S2]-LoweringEq}) and (\ref{[S2]-RaisingEq}).
This completes the proof of Theorem \ref{[S2]-THEO-4}.
\end{proof}



\begin{theorem}[Holonomic equation]
\label{[S2]-THEO-5} The Freud-type polynomial $Q_{n}^{t}(x)$ satisfies the
holonomic equation (second order linear differential equation)%
\begin{equation}
\mathcal{A}(x,t;n)[Q_{n}^{t}(x)]^{\prime \prime }+\mathcal{B}%
(x,t;n)[Q_{n}^{t}(x)]^{\prime }+\mathcal{C}(x,t;n)Q_{n}^{t}(x)=0,
\label{[S2]-2ndODE}
\end{equation}%
where%
\begin{eqnarray*}
\mathcal{A}(x,t;n) &=&\Psi _{1,1}(x,t;n)\Delta ^{2}(x,t;n), \\
\mathcal{B}(x,t;n) &=&\Delta (x,t;n)\left( W\left\{ \Psi
_{1,1}(x,t;n),\Delta (x,t;n)\right\} +\Psi _{1,1}(x,t;n)\big[\Psi
_{2,1}(x,t;n)-\Psi _{1,2}(x,t;n)\big]\right), \\
\mathcal{C}(x,t;n) &=&\Delta (x,t;n)W\left\{ \Psi _{1,1}(x,t;n),\Psi
_{2,1}(x,t;n)\right\} +\Psi _{1,1}(x,t;n)%
\begin{vmatrix}
\Psi _{1,1}(x,t;n) & \Psi _{1,2}(x,t;n) \\ 
\Psi _{2,1}(x,t;n) & \Psi _{2,2}(x,t;n)%
\end{vmatrix}%
,
\end{eqnarray*}%
with%
\begin{equation*}
\Psi _{i,j}(x,t;n)=%
\begin{vmatrix}
A_{i}(x,t;n) & C_{j}(x,t;n) \\ 
B_{i}(x,t;n) & D_{j}(x,t;n)%
\end{vmatrix}%
, \quad i,j=1,2.
\end{equation*}%
Moreover, the polynomials $\mathcal{A}(x,t;n)$ and $\mathcal{C}(x,t;n)$ are
even functions and the polynomial $\mathcal{B}(x,t;n) $ is an odd function,
whose coefficients are showed in Table \ref{Tabla1:coef}.
\end{theorem}

\begin{proof}
The proof of Theorem \ref{[S2]-THEO-5} comes directly from the ladder
operators given in Theorem \ref{[S2]-THEO-4}. The usual technique consists
in applying the raising operator to both sides of the equation satisfied by
the lowering operator, i.e.%
\begin{equation*}
\mathfrak{a}_{n}^{\dag }\left[ Q_{n-1}^{t}(x)\right] =\mathfrak{a}_{n}^{\dag
}\left[ \frac{1}{\Psi _{1,1}(x,t;n)}\mathfrak{a}_{n}[Q_{n}^{t}(x)]\right]
=\Psi _{2,2}(x,t;n)Q_{n}^{t}(x).
\end{equation*}%
Thus,%
\begin{equation*}
\frac{\Psi _{1,2}(x,t;n)}{\Psi _{1,1}(x,t;n)}\mathfrak{a}_{n}[Q_{n}^{t}(x)]-%
\Delta (x,t;n)\frac{d}{dx}\left( \frac{1}{\Psi _{1,1}(x,t;n)}\mathfrak{a}%
_{n}[Q_{n}^{t}(x)]\right) =\Psi _{2,2}(x,t;n)Q_{n}^{t}(x)
\end{equation*}%
which becomes%
\begin{equation*}
\begin{array}{lll}
\frac{\Psi _{1,2}(x,t;n)\Psi _{2,1}(x,t;n)}{\Psi _{1,1}(x,t;n)}%
[Q_{n}^{t}(x)]+\frac{\Psi _{1,2}(x,t;n)\Delta (x,t;n)}{\Psi _{1,1}(x,t;n)}%
[Q_{n}^{t}(x)]^{\prime } &  &  \\ 
-\Delta (x,t;n)\frac{d}{dx}\left( \frac{\Psi _{2,1}(x,t;n)}{\Psi
_{1,1}(x,t;n)}[Q_{n}^{t}(x)]+\frac{\Delta (x,t;n)}{\Psi _{1,1}(x,t;n)}%
[Q_{n}^{t}(x)]^{\prime }\right) & = & \Psi _{2,2}(x,t;n)Q_{n}^{t}(x).%
\end{array}%
\end{equation*}%
The intermediate computations above yield%
\begin{eqnarray*}
\frac{d}{dx}\left( \frac{\Psi _{2,1}(x,t;n)}{\Psi _{1,1}(x,t;n)}%
[Q_{n}^{t}(x)]\right) &=&%
\begin{vmatrix}
\Psi _{1,1}(x,t;n) & \Psi _{2,1}(x,t;n) \\ 
\Psi _{1,1}^{\prime }(x,t;n) & \Psi _{2,1}^{\prime }(x,t;n)%
\end{vmatrix}%
\frac{Q_{n}^{t}(x)}{\Psi _{1,1}^{2}(x,t;n)}+\frac{\Psi _{2,1}(x,t;n)}{\Psi
_{1,1}(x,t;n)}[Q_{n}^{t}(x)]^{\prime }, \\
\frac{d}{dx}\left( \frac{\Delta (x,t;n)}{\Psi _{1,1}(x,t;n)}%
[Q_{n}^{t}(x)]^{\prime }\right) &=&%
\begin{vmatrix}
\Psi _{1,1}(x,t;n) & \Delta (x,t;n) \\ 
\Psi _{1,1}^{\prime }(x,t;n) & \Delta ^{\prime }(x,t;n)%
\end{vmatrix}%
\frac{[Q_{n}^{t}(x)]^{\prime }}{\Psi _{1,1}^{2}(x,t;n)}+\frac{\Delta (x,t;n)%
}{\Psi _{1,1}(x,t;n)}[Q_{n}^{t}(x)]^{\prime \prime },
\end{eqnarray*}%
and 
\begin{eqnarray*}
&&\frac{\Psi _{1,2}(x,t;n)\Psi _{2,1}(x,t;n)}{\Psi _{1,1}(x,t;n)}%
[Q_{n}^{t}(x)]+\frac{\Psi _{1,2}(x,t;n)\Delta (x,t;n)}{\Psi _{1,1}(x,t;n)}%
[Q_{n}^{t}(x)]^{\prime } \\
&&-\frac{\Delta (x,t;n)}{\Psi _{1,1}^{2}(x,t;n)}%
\begin{vmatrix}
\Psi _{1,1}(x,t;n) & \Psi _{2,1}(x,t;n) \\ 
\Psi _{1,1}^{\prime }(x,t;n) & \Psi _{2,1}^{\prime }(x,t;n)%
\end{vmatrix}%
[Q_{n}^{t}(x)]-\frac{\Delta (x,t;n)\Psi _{2,1}(x,t;n)}{\Psi _{1,1}(x,t;n)}%
[Q_{n}^{t}(x)]^{\prime } \\
&&-\frac{\Delta (x,t;n)}{\Psi _{1,1}^{2}(x,t;n)}%
\begin{vmatrix}
\Psi _{1,1}(x,t;n) & \Delta (x,t;n) \\ 
\Psi _{1,1}^{\prime }(x,t;n) & \Delta ^{\prime }(x,t;n)%
\end{vmatrix}%
[Q_{n}^{t}(x)]^{\prime }-\frac{\Delta ^{2}(x,t;n)}{\Psi _{1,1}(x,t;n)}%
[Q_{n}^{t}(x)]^{\prime \prime } \\
&=&\Psi _{2,2}(x,t;n)Q_{n}^{t}(x).
\end{eqnarray*}%
Combining all the above expressions, and after some cumbersome computations,
Theorem \ref{[S2]-THEO-5} follows.
\end{proof}


\begin{corollary}
For every nonnegative integer $n$, when $n$ is odd \eqref{[S2]-2ndODE} is
equivalent to $\mathcal{A}(x,t;n)$ times \eqref{EcDiLi} and, in the other
case the polynomial coefficients of \eqref{[S2]-2ndODE} are contained in
Table \ref{Tabla2:coefeven}.
\end{corollary}



\begin{proof}
The procedure is to observe that in the case where the degree is odd, %
\eqref{[S2]-2ndODE} is reduced to the second order linear differential
equation \eqref{EcDiLi}, which is satisfied by $F_{2m+1}^{t}(x)$ since 
\begin{equation*}
\mathcal{B}(x,t;2m+1)=\mathcal{A}(x,t;2m+1)\,R_{2m+1}^{t}(x)\quad \text{and}%
\quad \mathcal{C}(x,t;2m+1)=\mathcal{A}(x,t;2m+1)\,S_{2m+1}^{t}(x).
\end{equation*}

On the other hand, if the degree is even then we get $b_{2m-1}^{t}=1$.
\end{proof}



For a deeper discussion of ladder operators we refer the reader to \cite[Ch.
3]{Ism05}. We next provide the second order linear differential equation
satisfied by the SMOP $\{Q_{n}^{t}\}_{n\geqslant 0}$ taking into account the
measure $\mu _{t}$ is semi-classical. This is the main tool for the further
electrostatic interpretation of zeros. Once we have the second order linear
differential equation satisfied by the SMOP $\{Q_{n}^{t}\}_{n\geqslant 0}$
it is easy to obtain an electrostatic model for their zeros. We will study
the asymptotic behavior of the position of the movable constant charges
involved in the external field. As we have shown in section \ref%
{[SECTION-3]-ConnForm} Freud-type orthogonal polynomials and Freud
polynomials of odd degree coincide. Thus, in this Section we shall derive
the electrostatic model for the zeros in the case when $n$ is even.

Let us evaluate \eqref{[S2]-2ndODE} at the zeros $\{y_{2m,i}(t)\}_{i=1}^{2m}$
of the polynomial $Q_{2m}^{t}(x)$, yielding 
\begin{eqnarray*}
{\frac{[Q_{2m}^{t}\big(y_{2m,i}(t)\big)]^{\prime \prime }}{[Q_{2m}^{t}\big(%
y_{2m,i}(t)\big)]^{\prime }}} &=&-{\frac{\mathcal{B}\big(y_{2m,i}(t),t;2m%
\big)}{\mathcal{A}\big(y_{2m,i}(t),t;2m\big)}} \\
&=&\frac{16b_{2m}^{t}y_{2m,i}^{3}(t)+8y_{2m,i}(t)\big[\big(%
a_{2m}^{2}(t)\,b_{2m}^{t}-t\big)b_{2m}^{t}+a_{2m+1}^{2}(t)\big]}{%
4b_{2m}^{t}y_{2m,i}^{4}(t)+4y_{2m,i}^{2}(t)\big[\big(a_{2m}^{2}(t)%
\,b_{2m}^{t}-t\big)b_{2m}^{t}+a_{2m+1}^{2}(t)\big]+h_{2m}^{t}} \\
&&-\frac{2}{y_{2m,i}(t)}+4\big(y_{2m,i}^{3}(t)-t\,y_{2m,i}(t)\big),
\end{eqnarray*}%
where%
\begin{equation*}
h_{2m}^{t}=\big(4a_{2m}^{2}(t)\big[a_{2m}^{2}(t)+a_{2m-1}^{2}(t)-t\big]\big[%
b_{2m}^{t}-1\big]-1\big)\big(b_{2m}^{t}-1\big).
\end{equation*}%
The above equation reads as the electrostatic equilibrium condition for $%
\{y_{{2m},i}(t)\}_{i=1}^{2m}$. Having 
\begin{equation}
u(x,t;2m)=4b_{2m}^{t}x^{4}+4x^{2}\big(\lbrack
a_{2m}^{2}(t)\,b_{2m}^{t}-t]b_{2m}^{t}+a_{2m+1}^{2}(t)\big)+h_{2m}^{t}
\label{u}
\end{equation}%
the previous condition can be rewritten as 
\begin{equation*}
\sum_{j=1,j\neq i}^{2m}{\frac{1}{y_{{2m},j}(t)-y_{{2m},i}(t)}}+{\frac{1}{2}}{%
\frac{[u]^{\prime }\big(y_{{2m},i}(t),t;{2m}\big)}{u\big(y_{{2m},i}(t),t;{2m}%
\big)}}-\frac{1}{y_{2m,i}(t)}+{2}\big(y_{{2m},i}^{3}(t)-ty_{{2m},i}(t)\big)%
=0,
\end{equation*}%
%
%
%
%
%
%
%
%
%
%
%
%
%
%
%
%
%
%
%
%
%
%
%
%
%
%
%
%
%
%
%
%
%
%
%
%
%
%
%
%
%
%
%
%
%
%
%
%
%
%
%
%
which means that the set of zeros $\{y_{{2m},i}(t)\}_{i=1}^{2m}$ are the
critical points (zeros of the gradient) of the total energy. Hence, the
electrostatic interpretation of the distribution of zeros means that we have
an equilibrium position under the presence of an external potential 
\begin{equation*}
V^{ext}(x)={\frac{1}{2}}\ln u(x,t;2m)-{\frac{1}{2}}\ln
x^{2}e^{-x^{4}+2tx^{2}},
\end{equation*}%
where the first term represents a \textit{short range potential}
corresponding to a unit charge located at the real zeros of the quartic
polynomial $u(x,t;2m)$, and the second one is a \textit{long range potential}
associated with the Freud weight function.



\section{Zeros of Freud-type SMOP}

\label{[SECTION-5]-Zeros}



Having in mind the techniques shown in \cite{IW-AJMS11}, we will first study
the motion of zeros of time depending of the polynomial $F_{n}^{t}(x)$
(Theorem \ref{motionzerosF}) and, finally, in subsection \ref%
{[SECTION-6.2]-EqMotion-FKrall} we give the differential equation that the
zeros of the polynomial $Q_{n}^{t}(x)$ satisfy.



\subsection{Equations of motion for zeros of $\{F_{n}^{t}(x)\}_{n\geqslant
0} $}

\label{[SECTION-6.1]-EqMotion-F}



Notice that the three term recurrence relation \eqref{3TRR-Monic} implies
that 
\begin{equation}
x^{2}\,F_{n}^{t}(x)=F_{n+2}^{t}(x)+\big[a_{n+1}^{2}(t)+a_{n}^{2}(t)\big]%
F_{n}^{t}(x)+a_{n}^{2}(t)\,a_{n-1}^{2}(t)\,F_{n-2}^{t}(x)  \label{x2}
\end{equation}%
for all $n\geqslant 1$ and $F_{-1}^{t}(x)=0$. The requirement on ${\frac{%
\partial F_{n}^{t}}{\partial t}}(x)$ is that 
\begin{equation*}
{\frac{\partial F_{n}^{t}}{\partial t}}(x)=\sum_{i=0}^{n-2}\flat
_{n,i}\,F_{i}^{t}(x),
\end{equation*}%
with $\displaystyle\flat _{n,i}={\frac{\left\langle {\frac{\partial F_{n}^{t}%
}{\partial t}},F_{i}^{t}\right\rangle _{t}}{\left\Vert F_{i}^{t}\right\Vert
_{t}^{2}}}$. From 
\begin{equation*}
0=\int_{-\infty }^{\infty }F_{n}^{t}(x)\,F_{i}^{t}(x)\,\omega
_{t}(x)dx,\quad 0\leqslant i\leqslant n-2,
\end{equation*}%
differentiating the above expression with respect to $t$, we obtain 
\begin{equation*}
0=\int_{-\infty }^{\infty }\left[ F_{i}^{t}(x)\,{\frac{\partial F_{n}^{t}}{%
\partial t}}(x)+F_{n}^{t}(x){\frac{\partial F_{i}^{t}}{\partial t}}(x)\right]
\omega _{t}(x)dx+2\int_{-\infty }^{\infty
}x^{2}\,F_{n}^{t}(x)\,F_{i}^{t}(x)\,\omega _{t}(x)dx.
\end{equation*}%
Hence, since $\displaystyle{\frac{\partial F_{n}^{t}}{\partial t}}(x)$ is a
linear combination of the first $n$ elements of the sequence $%
\{F_{k}^{t}\}_{k\geqslant 0}$, using \eqref{x2} we get 
\begin{equation*}
\left\langle {\frac{\partial F_{n}^{t}}{\partial t}},F_{i}^{t}\right\rangle
_{t}=-2a_{n}^{2}(t)\,a_{n-1}^{2}(t)\,\left\Vert F_{n-2}^{t}\right\Vert
_{t}^{2}\,\delta _{n-2,i},\quad 0\leqslant i\leqslant n-2,
\end{equation*}%
and 
\begin{equation}
{\frac{\partial F_{n}^{t}}{\partial t}}(x)=-2a_{n}^{2}(t)\,a_{n-1}^{2}(t)%
\,F_{n-2}^{t}(x)  \label{3.7mehi}
\end{equation}%
as claimed for $t>0$ and $n\geqslant 1$ with $F_{-1}^{t}(x)=0$.

We can now formulate this result



\begin{theorem}[Equation of motion for zeros of $F_{n}^{t}(x)$]
\label{motionzerosF} Let $n$ be a positive integer and $t>0$. If $%
x_{n,1}(t),\dots ,x_{n,n}(t)$ are the $n$ zeros of $F_{n}^{t}(x)$, then 
\begin{equation*}
{\frac{\partial x_{n,k}(t)}{\partial t}}={\frac{x_{n,k}(t)}{2\left(
x_{n,k}^{2}(t)-t+a_{n}^{2}(t)+a_{n+1}^{2}(t)\right) }}.
\end{equation*}
\end{theorem}



\begin{proof}
Given $n\geqslant1$ and $t>0$, we have 
\begin{equation}  \label{3.9mehi}
F_n^t\big(x_{n,k}(t)\big)=0,\quad 1\leqslant k\leqslant n.
\end{equation}
If we differentiate the above equation with respect to time $t$, we get 
\begin{eqnarray}  \label{=0}
{\frac{\partial x_{n,k}(t) }{\partial t}}\,{\frac{\partial F_n^t}{\partial x}%
}(x) \bigg\rvert_{x=x_{n,k}(t)}+{\frac{\partial F_n^t}{\partial t}}\big(%
x_{n,k}(t)\big)=0.
\end{eqnarray}

On the other hand, from \eqref{[S2]-StructRelation} and \eqref{3.9mehi}, we
obtain 
\begin{equation*}
{\frac{\partial F_n^t}{\partial x}}(x)\bigg\rvert_{x=x_{n,k}(t)}= b\big(%
x_{n,k}(t),t;n\big)\,F_{n-1}^t\big(x_{n,k}(t)\big).
\end{equation*}
Hence, using \eqref{3.7mehi}, \eqref{=0} and the above formula, it follows
that 
\begin{equation}  \label{last}
{\frac{\partial x_{n,k}(t) }{\partial t}}= {\frac{a_{n-1}^2(t)\,F_{n-2}^t%
\big(x_{n,k}(t)\big)}{2\left(
x_{n,k}^{2}(t)-t+a_{n}^{2}(t)+a_{n+1}^{2}(t)\right)\,F_{n-1}^t\big(x_{n,k}(t)%
\big)}}.
\end{equation}
Finally, we obtain the result since the three term recurrence relation %
\eqref{3TRR-Monic} yields 
\begin{equation*}
F_{n-2}^t\big(x_{n,k}(t)\big)={\frac{x_{n,k}(t)}{a_{n-1}^2(t)}}\,F_{n-1}^t%
\big(x_{n,k}(t)\big).
\end{equation*}
\end{proof}



\subsection{Equations of motion for zeros of Freud-type orthogonal
polynomials}

\label{[SECTION-6.2]-EqMotion-FKrall}



It is required that 
\begin{equation*}
{\frac{\partial Q_{n}^{t}}{\partial t}}(y)=\sum_{i=0}^{n-2}\widetilde{\flat }%
_{n,i}\,Q_{i}^{t}(y),
\end{equation*}%
with $\displaystyle\widetilde{\flat }_{n,i}={\frac{\left\langle {\frac{%
\partial Q_{n}^{t}}{\partial t}},Q_{i}^{t}\right\rangle }{\left\Vert
Q_{i}^{t}\right\Vert ^{2}}}$. We first compute the coefficients $\widetilde{%
\flat }_{n,i}$, $0\leqslant i\leqslant n-2$. From the orthogonality
relations 
\begin{equation*}
0=\int_{-\infty }^{\infty }Q_{n}^{t}(y)\,Q_{i}^{t}(y)\,\omega
_{t}(y)dy+MQ_{n}^{t}(0)Q_{i}^{t}(0),\quad 0\leqslant i\leqslant n-2.
\end{equation*}%
Taking the partial derivative with respect to the variable $t$ in the above
equation we obtain, for $0\leqslant i\leqslant n-1$, 
\begin{equation*}
\left\langle {\frac{\partial Q_{n}^{t}}{\partial t}}(y),Q_{i}^{t}(y)\right%
\rangle =-2\left\langle Q_{n}^{t}(y),y^{2}Q_{i}^{t}(y)\right\rangle _{t}.
\end{equation*}%
Applying \eqref{Kn-PropRepro} and \eqref{CF-1} in the above equality we get 
\begin{equation*}
\widetilde{\flat }_{n,i}=0,\quad 0\leqslant i\leqslant n-3,
\end{equation*}%
since 
\begin{equation*}
\left\langle Q_{n}^{t}(y),y^{2}Q_{i}^{t}(y)\right\rangle _{t}=\left\langle
F_{n}^{t}(y),y^{2}Q_{i}^{t}(y)\right\rangle _{t}-MQ_{n}^{t}(0)\left\langle
K_{n-1}(x,0;t),y^{2}Q_{i}^{t}(y)\right\rangle _{t}.
\end{equation*}%
Then, we get 
\begin{equation*}
{\frac{\partial Q_{n}^{t}}{\partial t}}(y)=\widetilde{\flat }%
_{n,n-2}\,Q_{n-2}^{t}(y),
\end{equation*}%
where 
\begin{equation*}
\widetilde{\flat }_{n,n-2}=-{\frac{2}{b_{n-2}^{t}}}\left(
a_{n}^{2}(t)a_{n-1}^{2}(t)+{\frac{M[F_{n}^{t}(0)]^{2}}{\big(1+MK_{n-1}(0,0;t)%
\big)\left\Vert {F_{n-2}^{t}}\right\Vert _{t}^{2}}}\right).
\end{equation*}


Let us denote by $y_{n,1}(t),\dots ,y_{n,n}(t)$ the $n$ zeros of polynomial $%
Q_{n}^{t}(y)$. From \eqref{RR3T-Tilde} it follows that

\begin{equation}  \label{dotQ}
{\frac{\partial Q_{n}^{t}}{\partial t}}\big(y_{n,k}(t)\big)=\widetilde{\flat 
}_{n,n-2}{\frac{y_{n,k}(t)\,b_{n-2}^{t}}{b_{n-1}^{t}\,a_{n-1}^{2}(t)}}%
Q_{n-1}^{t}\big(y_{n,k}(t)\big).
\end{equation}

We can now state the analogue of Theorem \ref{motionzerosF} for the zeros of
Freud-type polynomials.



\begin{theorem}[Equation of motion for zeros of $Q_{n}^{t}(y)$]
\label{motionzerosQ} Let $n$ be a positive integer and $t>0$. If $%
y_{n,1}(t),\dots ,y_{n,n}(t)$ are the $n$ zeros of $Q_{n}^{t}(y)$, then 
\begin{equation*}
{\frac{\partial y_{n,k}}{\partial t}}(t)={\frac{-y_{n,k}^{2}(t)\,\widetilde{%
\flat }_{n,n-2}{\frac{b_{n-2}^{t}}{b_{n-1}^{t}\,a_{n-1}^{2}(t)}}Q_{n-1}^{t}%
\big(y_{n,k}(t)\big)}{C_{k,1}(n,t)F_{n}^{t}\big(y_{n,k}(t)\big)%
+C_{k,2}(n,t)F_{n-1}^{t}\big(y_{n,k}(t)\big)}},
\end{equation*}%
where 
\begin{eqnarray*}
C_{k,1}(n,t) &=&1-4a_{n}^{2}(t)\big[b_{n}^{t}y_{n,k}^{2}(t)+\big(b_{n}^{t}-1%
\big)\big(a_{n}^{2}(t)+a_{n-1}^{2}(t)-t\big)\big], \\
C_{k,2}(n,t) &=&4a_{n}^{2}(t)\,y_{n,k}(t)\left[ b_{n}^{t}\big(%
y_{n,k}^{2}(t)-t+a_{n}^{2}(t)\big)+a_{n+1}^{2}(t)\right] .
\end{eqnarray*}
\end{theorem}



\begin{proof}
Given $n\geqslant1$ and $t>0$, we have 
\begin{equation}  \label{3.9mehiQ}
Q_n^t\big(y_{n,k}(t)\big)=0,\quad 1\leqslant k\leqslant n.
\end{equation}
If we differentiate the above equation with respect to $t$, we get 
\begin{eqnarray}  \label{Q=0}
{\frac{\partial y_{n,k}}{\partial t}}(t)\,{\frac{\partial Q_n^t}{\partial y}}%
(y) \bigg\rvert_{y=y_{n,k}(t)}+{\frac{\partial Q_n^t}{\partial t}}\big(%
y_{n,k}(t)\big)=0.
\end{eqnarray}


We only need to compute the $y$ derivative of $Q_n^t(y)$ and to combine %
\eqref{dotQ} with the above expression in order to get the result.
Differentiating \eqref{[S3]-hxConnForm} and applying \eqref{3TRR-Monic}, %
\eqref{[S2]-StructRelation}, and \eqref{3.9mehiQ}, we have 
\begin{eqnarray*}
y_{n,k}(t)\,{\frac{\partial Q_n^t}{\partial y}}(y)\bigg\rvert%
_{y=y_{n,k}(t)}&=& C_{k,1}(n,t)F_{n}^{t}\big(y_{n,k}(t)\big)+
C_{k,2}(n,t)F_{n-1}^{t}\big(y_{n,k}(t)\big).
\end{eqnarray*}%
Hence, combining \eqref{dotQ} and the above formula with \eqref{Q=0} the
Theorem follows.
\end{proof}



In the case when $n$ is odd, the above Theorem provides \eqref{last}, which
establishes the formula of Theorem \ref{motionzerosF}.



\subsection{Behavior and monotonicity with $M$ of the zeros of ${%
Q_{2m}^{t}(x)}$}

\label{[SECTION-6.3]-Monotonicity}



Let assume that $y_{n,k}$, $k=1,2,...,n,$ are the zeros of $Q_{n}^{t}(x)$
arranged in an increasing order. From the analysis done before, it is clear
that the zeros $y_{n,s}$ when $n$ is odd are not affected by the mass $M$.
Next, we analyze the behavior of zeros $y_{2m,s}=y_{2m,s}(M)$, $s=1,\ldots
,2m,$ as a function of the mass $M$ and we obtain such a behavior when the
positive real number $M$ goes from zero to infinity. In order to do that, we
use a technique developed in \cite[Lemma 1]{BDR-JCAM02} and \cite[Lemmas 1
and 2]{DMR-ANM10} concerning the behavior and the asymptotics of the zeros
of linear combinations of two $n$-th degree polynomials $h_{n},g_{n}\in 
\mathbb{P}$ with interlacing zeros, such that $f(x)=h_{n}(x)+cg_{n}(x)$, $%
c\geqslant 0$. From now on, we will refer to this technique as the \textit{%
Interlacing Lemma}. Here the linear combination of two polynomials of the
same degree $2m$ is given by (\ref{[S2]-ConnForm-Main}), and\ $F_{2m}^{t}$, $%
G_{2m}$ play the role of $h_{n}(x)$, $g_{n}(x)$ respectively.

In order to apply this technique, we need to show that the hypotheses of the
Interlacing Lemma are fulfilled. First, in Theorem \ref{T1} the interlacing
of the zeros of $F_{2m}^{t}$ and $G_{2m}$ was proved. In our computations,
we will only deal with the zero behavior in the positive real semi-axis,
because the behavior in $\mathbb{R}_{-}$ follows by reflection through the $%
y $-axis by symmetry reasons as usual. Thus, from (\ref{[S2]-ConnForm-Main}%
), the positivity of $K_{2m-1}(0,0;t)$, and Theorem \ref{T1} we are in the
hypothesis of the Interlacing Lemma, and we immediately conclude the
following results about monotonicity, asymptotics, and speed of convergence
for the zeros of $Q_{2m}^{t}(x)$ in terms of the mass $M$.

Let us define the monic polynomials%
\begin{eqnarray*}
G_{m}^{l}(x) &=&x(x-g_{2m,1})(x-g_{2m,2})\cdots (x-g_{2m,m-1}) \\
&=&x(x-x_{2m-1,1}^{[2]})(x-x_{2m-1,2}^{[2]})\cdots (x-x_{2m-1,m-1}^{[2]})
\end{eqnarray*}%
and%
\begin{eqnarray*}
G_{m}^{r}(x) &=&x(x-g_{2m,m+2})(x-g_{2m,m+3})\cdots (x-g_{2m,2m}) \\
&=&x(x-x_{2m-1,m+1}^{[2]})(x-x_{2m-1,m+2}^{[2]})\cdots
(x-x_{2m-1,2m-1}^{[2]}),
\end{eqnarray*}%
such that $G_{2m}^{l}(x)=G_{m}^{l}(x)G_{m}^{r}(x)$.



\begin{theorem}
\label{InterlQKrall} In the negative real semiaxis, the following
interlacing property holds 
\begin{equation*}
x_{2m,1}<y_{2m,1}<g_{2m,1}<x_{2m,2}<y_{2m,2}<\cdots
<g_{2m,m-1}<x_{2m,m}<y_{2m,m}<g_{2m,m}=0.
\end{equation*}%
Moreover, each $y_{n,l}=y_{n,l}(M)$ is an increasing function of $M$ and,
for each $l=1,\ldots ,m$,%
\begin{equation*}
\lim_{M\rightarrow \infty }y_{2m,l}(M)=g_{2m,l}\,,
\end{equation*}%
as well as%
\begin{equation*}
\lim\limits_{M\rightarrow \infty }M[g_{2m,l}-y_{2m,l}]=\dfrac{%
F_{2m}^{t}(g_{2m,l})}{K_{2m-1}(0,0;t)[G_{m}^{l}]^{\prime }(g_{2m,l})}.
\end{equation*}%
Applying symmetry properties through y-axis, in the positive real semiaxis
the following interlacing property holds 
\begin{equation*}
0=g_{2m,m+1}<y_{2m,m+1}<x_{2m,m+1}<g_{2m,m+2}<\cdots
<x_{2m,2m-1}<g_{2m,2m}<y_{2m,2m}<x_{2m,2m}.
\end{equation*}%
Moreover, each $y_{n,r}=y_{n,r}(M)$ is a decreasing function of $M$ and, for
each $r=m+1,\ldots ,2m$,%
\begin{equation*}
\lim_{M\rightarrow \infty }y_{2m,r}=g_{2m,r}\,,
\end{equation*}%
as well as%
\begin{equation*}
\lim\limits_{M\rightarrow \infty }M[y_{2m,r}-g_{2m,r}]=\dfrac{%
-F_{2m}^{t}(g_{2m,r})}{K_{2m-1}(0,0;t)[G_{m}^{r}]^{\prime }(g_{2m,r})}.
\end{equation*}
\end{theorem}



Notice that the mass point at $x=0$ attracts two zeros of $Q_{2m}^{t}(x)$,
i.e. when $M\rightarrow \infty $, it captures $y_{2m,m}$ and $y_{2m,m+1}$ at
the same time.



\section{Numerical experiments}

\label{[SECTION-6]-NumExp}



We next provide some numerical experiments using Mathematica$^{\circledR }$
software, dealing with the zeros of Freud-type polynomials $%
\{Q_{n}^{t}\}_{n\geqslant 0}$. More specifically, we will show the position
of the two symmetric and closest-to-the-origin zeros of some even
polynomials of the sub-sequence $\{Q_{2k}^{t}\}_{k\geqslant 0}\,$. We choose 
$Q_{4}^{t}(x)$ for the following first experiment varying $t$. We show the
location of its zeros for several values of $t$ and $M$, and we also show
the position of the source-charges of the short range potential 
\begin{equation*}
\upsilon _{short}(x)=\frac{1}{2}\ln u(x,t;2m),\quad m=1,2,3,\ldots ,
\end{equation*}
which are the zeros of the polynomial $u(x,t;4)$ defined in \eqref{u}. In
tables \ref{Tabla3}, \ref{Tabla4}, \ref{Tabla5} and \ref{Tabla6} we provide
numerical evidence of the position of its zeros when $t$ is equal to $%
1/2,1,3/2,$ and $2$ respectively, for several choices of $M$. Notice that
the polynomial $u(x,t;4)$ has exactly degree four, and its zeros are always
two real and two simple conjugate complex numbers. We also remark that we
recover the results in \cite[Cor. 3.5]{GAM-ETNA05} when $t=0$.



Figure \ref{PicFreud} illustrates the change in the even Freud-type
polynomials when $M$ varies as described in Theorem \ref{InterlQKrall}. We
enclose the graphs of $Q_{4}^{1}(x)$ for three different values of $M$. The
black continuous, dashed, and dotted lines correspond to $M=0$, $M=0.2,$ and 
$M=0.6$, respectively. We also include, with different tones of gray color,
the graphs of $Q_{3}^{1}(x)$ and $Q_{5}^{1}(x)$, showing that the odd degree
polynomials are not affected by the variation of the mass $M$.


Finally, the last two tables \ref{Tabla7} and \ref{Tabla8} show the position
of the zeros of Freud-type polynomials of $Q_{6}^{1}(x)$ and $Q_{10}^{1}(x)$
and the zeros of the corresponding ghost polynomials. Notice that the zeros
of the ghost polynomials continue to be two real and two complex conjugate
numbers.


\section*{Acknowledgements}

The work of the second and third authors was partially supported by Direcció%
n General de Investigación Científica y Técnica, Ministerio de Economía y
Competitividad of Spain, grant MTM2012-36732-C03-01. 



\clearpage



\begin{figure}[ht]
\centerline{\includegraphics[width=11cm,keepaspectratio]{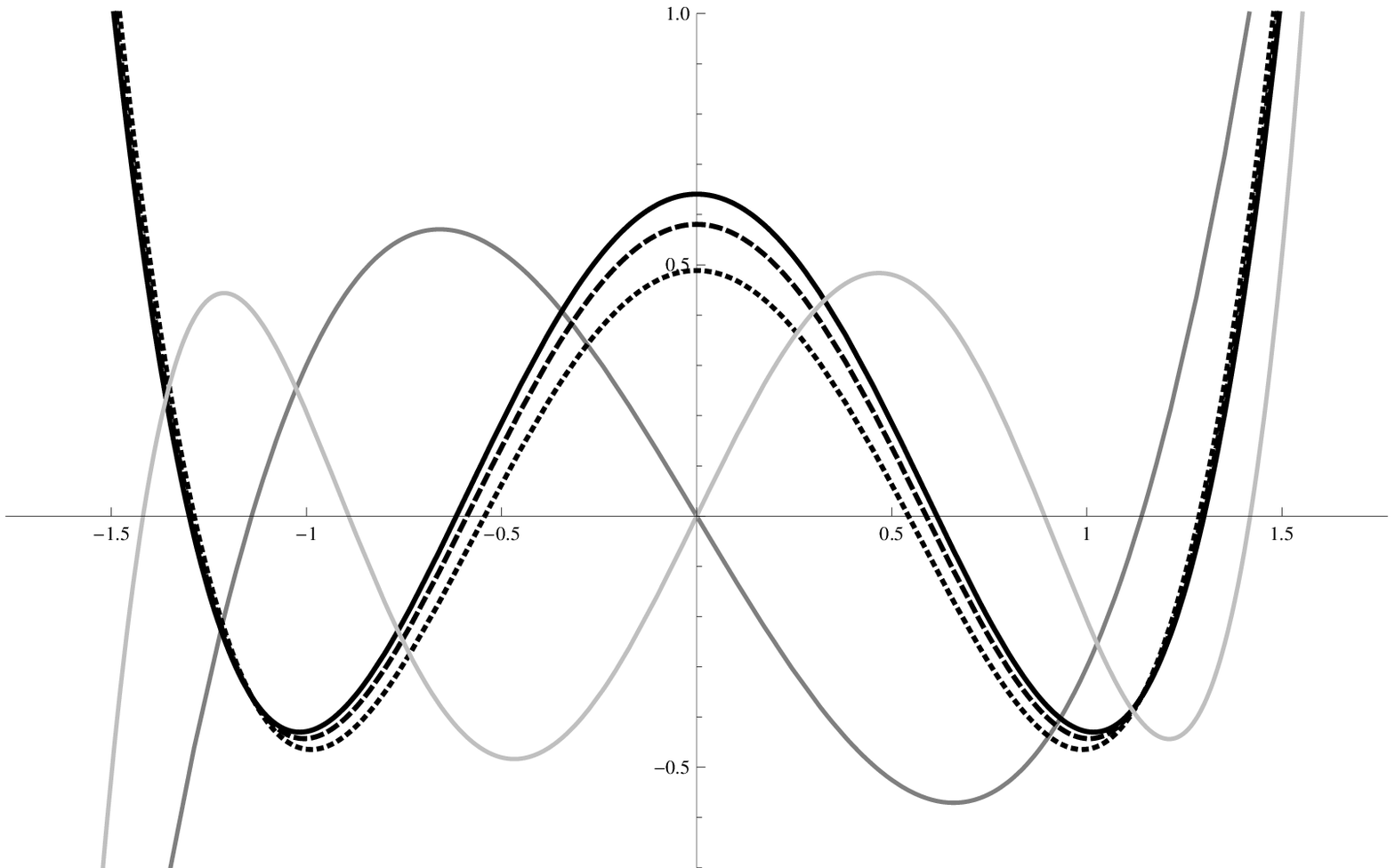}}
\caption{The graphs of $Q_{3}^{1}(x)$ and $Q_{5}^{1}(x)$ (gray) and $%
Q_{4}^{1}(x)$ for some values of $M$ (black lines).}
\label{PicFreud}
\end{figure}



\begin{table}[ht]
\centering
{\footnotesize 
\begin{tabular}{|c|c|c|}
\hline
Polynomial & Power & Coefficient \\[.3ex] \hline
&  &  \\[-1.2ex] 
& 8 & $4 a_n^2(t)\, b_n^t\big[b_{n-1}^t\big]^2 $ \\[1.5ex] 
& 6 & $4 a_n^2(t) \,b_{n-1}^t \bigl\{ b_n^t \left(a_n^2(t) \big[2(1-
b_{n}^t)+b_{n-1}^t (3 b_{n}^t-2)\big]-t
b_{n-1}^t\right)+a_{n+1}^2(t)b_{n-1}^t\bigr\}$ \\[2ex] 
& 4 & $a_n^2(t) \big(b_{n}^t-1\big) \left\{4b_{n-1}^t a_n^2(t)
\left(b_{n-1}^t \left[ (b_{n}^t-1) a_{n-1}^2(t) + (1-3b_{n-1}^t)t\right]%
+2(b_{n-1}^t-1)a_{n+1}^2(t)\right.\right.$ \\[1ex] 
$\mathcal{A}(x,t;n)$ &  & $\left.\left.+2 t b_{n}^t \right)+4 a_n^4(t) \left[%
\big(3 [b_{n}^t]^2-1\big)
                           \lbrack b_{n-1}^t]^2+2b_{n}^tb_{n-1}^t (1-2
b_{n}^t)+(b_{n}^t-1) b_{n}^t\right]-[b_{n-1}^t]^2\right\}$ \\[2ex] 
& 2 & $2 a_n^4(t) \big(b_{n-1}^t-1\big)\big(b_{n}^t-1\big)^2 \left\{2
a_n^2(t)\left[b_{n}^t \left(t-b_{n}^ta_n^2(t) \right)+b_{n-1}^t \left(2
(b_{n}^t-1)a_{n-1}^2(t)\right.\right.\right.$ \\[1ex] 
&  & $\left.\left.\left.+a_n^2(t) \big[b_{n}^t (b_{n}^t+2)-2\big]+t (2-3
b_{n}^t)\right)+a_{n+1}^2(t) \big(b_{n-1}^t-1\big)
 \right]-b_{n-1}^t\right\}$ \\[2ex] 
& 0 & $-a_n^6(t) \big(b_{n-1}^t-1\big)^2 \big(b_{n}^t-1\big)%
^3\left\{1-4a_n^2(t) \big[b_{n}^t-1\big] \left[a_{n-1}^2(t)+a_n^2(t)-t\right]%
\right\}$ \\[1ex] \hline
&  &  \\[-1.2ex] 
& 11 & $-16 a_n^2(t)\, b_{n}^t [b_{n-1}^t]^2$ \\[1.5ex] 
& 9 & $-16 a_n^2(t)\,b_{n-1}^t \big\{b_{n}^t \big(a_n^2(t) \big[b_{n-1}^t (3
b_{n}^t-2)+2(1-b_{n}^t)\big]-2 t b_{n-1}^t\big)
        +b_{n-1}^ta_{n+1}^2(t) \big\}$ \\[2ex] 
& 7 & $-4a_n^2(t) \left\{[b_{n-1}^t]^2 [1-4 t a_{n+1}^2(t)+4 t^2
b_{n}^t+b_{n}^t] +4 b_{n-1}^t a_{n}^2(t) \left(b_{n-1}^t\left[
(b_{n}^t-1)^2a_{n-1}^2(t)\right.\right.\right.$ \\[1ex] 
&  & $\left.\left.+t (6 (1-b_{n}^t) b_{n}^t-1)\right]+2 (b_{n-1}^t-1)
(b_{n}^t-1)a_{n+1}^2(t)+4 t (b_{n}^t-1) b_{n}^t\right)$ \\[1ex] 
&  & $\left.+ 4 (b_{n}^t-1)a_{n}^4(t) \left[\left(3 [b_{n}^t]^2-1\right)
[b_{n-1}^t]^2+2 b_{n}^t (1-2 b_{n}^t) b_{n-1}^t+(b_{n}^t-1) b_{n}^t\right]%
\right\}$ \\[2ex] 
& 5 & $-4 (b_{n}^t-1)a_n^2(t) \left\{ 2 b_{n-1}^ta_n^2(t) \left(b_{n-1}^t
\left(-2 t(b_{n}^t-1)a_{n-1}^2(t) +(6 t^2 +1)b_{n}^t-2
t^2+1\right)\right.\right.$ \\[1ex] 
&  & $\left. -4 t a_{n+1}^2(t) (b_{n-1}^t-1)-\left(4 t^2+1\right)
b_{n}^t-1\right) +4 a_n^4(t) \left(2 b_{n-1}^t \left(t (1-2
b_{n}^t)^2\right.\right.$ \\[1ex] 
$\mathcal{B}(x,t;n)$ &  & $\left.-a_{n-1}^2(t)
(b_{n}^t-1)^2\right)-[b_{n-1}^t]^2 \left(-2 a_{n-1}^2(t) (b_{n}^t-1)^2+t
b_{n}^t (6 b_{n}^t-5)+t\right)$ \\[1ex] 
&  & $\left.+a_{n+1}^2(t) (b_{n}^t-1)(b_{n-1}^t-1)^2-2 t (b_{n}^t-1)
b_{n}^t\right)+4 a_n^6(t) (b_{n-1}^t-1) (b_{n}^t-1) $ \\[1ex] 
&  & $\left.\left(b_{n-1}^t (b_{n}^t (b_{n}^t+2)-2)-[b_{n}^t]^2\right)+t
[b_{n-1}^t]^2\right\}$ \\[2ex] 
& 3 & $-2 a_n^2(t) \big(b_{n}^t-1\big) \left\{2 a_n^4(t) [b_{n}^t-1]
\left(-2 b_{n-1}^t \left(-4 t a_{n-1}^2(t) [b_{n}^t-1]\right.\right.\right.$
\\[1ex] 
&  & $\left.+8 t^2 b_{n}^t+b_{n}^t-4 t^2+1\right)+[b_{n-1}^t]^2 \left(-8 t
a_{n-1}^2(t) [b_{n}^t-1]+12 t^2 b_{n}^t+b_{n}^t-8 t^2-1\right)$ \\[1ex] 
&  & $\left. -4 t a_{n+1}^2(t) [b_{n-1}^t-1]^2+4 t^2
b_{n}^t+b_{n}^t+1\right)-8 a_n^6(t) [b_{n-1}^t-1] [b_{n}^t-1] $ \\[1ex] 
&  & $\left(-[b_{n-1}^t-1] \left(a_{n-1}^2(t) [b_{n}^t-1]^2-2 t
[b_{n}^t]^2\right) -t [b_{n-1}^t-2 b_{n}^t+1]\right) +4 a_n^2(t) \,b_{n-1}^t
[b_{n}^t-1]$ \\[1ex] 
&  & $\left. \left(b_{n-1}^t [2 t-a_{n-1}^2(t)]-t\right)+8 a_n^8(t)
[b_{n-1}^t-1]^2 [b_{n}^t-1]^3+[b_{n-1}^t]^2\right\}$ \\[2ex] 
& 1 & $-4 a_n^4(t)\big(b_{n-1}^t-1\big) \big(b_{n}^t-1\big)^2 \left(1-4
a_n^2(t) [b_{n}^t-1] [a_{n-1}^2(t)+a_n^2(t)-t]\right) $ \\[1ex] 
&  & $\left(t a_n^2(t) [b_{n-1}^t-1] [b_{n}^t-1]+b_{n-1}^t\right)$ \\%
[1ex] \hline
\end{tabular}
}
\caption{Coefficients of the polynomials $\mathcal{A}(x,t;n)$ and $\mathcal{B%
}(x,t;n)$ for every integer $n$.}
\label{Tabla1:coef}
\end{table}



\begin{table}[ht]
\centering
\begin{tabular}{|c|c|c|}
\hline
Polynomial & Power & Coefficient \\[.3ex] \hline
&  &  \\[-1.2ex] 
& 8 & $4 a_{2n}^2(t)\, b_{2n}^t$ \\[1.5ex] 
& 6 & $4a_{2n}^2(t)\big(a_{{2n}+1}^{2}(t)+b_{2n}^t\big[a_{{2n}%
}^{2}(t)\,b_{2n}^t-t\big]\big)$ \\[2ex] 
$\mathcal{A}(x,t;2n)$ & 4 & $-a_{2n}^2(t)\big(1-4a_{{2n}}^{2}(t)\big[a_{{2n}%
}^{2}(t)+a_{{2n}-1}^{2}(t)-t\big]\big[b_{2n}^t-1\big]\big)\big(b_{2n}^t-1%
\big)$ \\[2ex] 
& 2 & $0$ \\[2ex] 
& 0 & $0$ \\[1ex] \hline
&  &  \\[-1.2ex] 
& 11 & $-16 a_{2n}^2(t)\, b_{{2n}}^t $ \\[1.5ex] 
& 9 & $-16a_{2n}^{2}(t)\big(a_{{2n}+1}^{2}(t)+b_{2n}^t\big[a_{{2n}%
}^{2}(t)\,b_{2n}^t-2t\big]\big)$ \\[2ex] 
& 7 & $-4a_{2n}^{2}(t)\big\{1-4ta_{{2n}+1}^{2}(t)+4a_{{2n}%
}^{4}(t)[b_{2n}^t-1]^2+b_{2n}^t+4t^2b_{2n}^t$ \\[1ex] 
&  & $+4a_{{2n}}^{2}(t)\big(a_{{2n}-1}^{2}(t)\big(b_{2n}^t-1\big)^2-t\big[1+2%
\big(b_{2n}^t-1\big)b_{2n}^t\big]\big)\big\}$ \\[2ex] 
$\mathcal{B}(x,t;2n)$ & 5 & $-4ta_{2n}^{2}(t)\big(1-4a_{{2n}}^{2}(t)\big[a_{{%
2n}}^{2}(t)+a_{{2n}-1}^{2}(t)-t\big]
      \big[b_{2n}^t-1\big]\big)\big(b_{2n}^t-1\big)$ \\[2ex] 
& 3 & $-2a_{2n}^{2}(t)\big(1-4a_{{2n}}^{2}(t)\big[a_{{2n}}^{2}(t)+a_{{2n}%
-1}^{2}(t)-t\big]\big[b_{2n}^t-1\big]\big)\big(b_{2n}^t-1\big)$ \\[2ex] 
& 1 & $0$ \\[1ex] \hline
\end{tabular}%
\caption{Coefficients of the polynomials $\mathcal{A}(x,t;2n)$ and $\mathcal{%
B}(x,t;2n)$ for every integer $n$.}
\label{Tabla2:coefeven}
\end{table}



\begin{table}[ht]
\centering
\renewcommand{\arraystretch}{1.3} {\footnotesize 
\begin{tabular}{@{}rrrrrrrrrrr}
\toprule & \multicolumn{7}{c}{$t = 0.5$} & \phantom{} &  &  \\ 
\cmidrule{2-8} & $M=0$ & $M=0.002$ & $M=0.05$ & $M=0.5$ & $M=5$ & $M=10$ & $%
M=50$ &  &  &  \\ 
\midrule $Q^{t}_4(x)$ &  &  &  &  &  &  &  &  &  &  \\ 
& $\pm$ 1.1640 & $\pm$1.1639 & $\pm$1.1623 & $\pm$1.1516 & $\pm$1.1318 & $%
\pm $1.1286 & $\pm 1.1257$ &  &  &  \\ 
& $\pm$ 0.4839 & $\pm$0.4836 & $\pm$0.4755 & $\pm$0.4154 & $\pm$0.2256 & $%
\pm $0.1689 & $\pm0.0794$ &  &  &  \\ 
$u(x,t;4)$ &  &  &  &  &  &  &  &  &  &  \\ 
& $\pm0.94861 i$ & $\pm0.94869 i$ & $\pm 0.9505i$ & $\pm0.9618i$ & $%
\pm0.9827 i$ & $\pm0.9863i$ & $\pm 0.9898i$ &  &  &  \\ 
& 0 & $\pm$ 0.0144 & $\pm$0.0689 & $\pm$0.1528 & $\pm$0.1219 & $\pm$0.0942 & 
$\pm 0.0455$ &  &  &  \\ 
\bottomrule &  &  &  &  &  &  &  &  &  & 
\end{tabular}
}
\caption{Zeros of $Q^{0.5}_4(x)$ and $u(x,0.5;4)$ for some values of $M$.}
\label{Tabla3}
\end{table}



\begin{table}[ht]
\centering
\renewcommand{\arraystretch}{1.3} {\footnotesize 
\begin{tabular}{@{}rrrrrrrrrrr}
\toprule & \multicolumn{7}{c}{$t = 1$} & \phantom{} &  &  \\ 
\cmidrule{2-8} & $M=0$ & $M=0.002$ & $M=0.05$ & $M=0.5$ & $M=5$ & $M=10$ & $%
M=50$ &  &  &  \\ 
\midrule $Q^{t}_4(x)$ &  &  &  &  &  &  &  &  &  &  \\ 
& $\pm$1.3002 & $\pm$1.3001 & $\pm$1.2988 & $\pm$1.2891 & $\pm$1.2659 & $\pm$%
1.2615 & $\pm 1.2570$ &  &  &  \\ 
& $\pm$0.6156 & $\pm$0.6153 & $\pm$0.6084 & $\pm$0.5533 & $\pm$0.3335 & $\pm$%
0.2551 & $\pm 0.1227$ &  &  &  \\ 
$u(x,t;4)$ &  &  &  &  &  &  &  &  &  &  \\ 
& $\pm0.7653i$ & $\pm0.7654i$ & $\pm0.7687i$ & $\pm0.7894i$ & $\pm 0.8354i$
& $\pm0.8454 i$ & $\pm0.8563i$ &  &  &  \\ 
& 0 & $\pm$0.0171 & $\pm$0.0827 & $\pm0.1962$ & $\pm 0.1800$ & $\pm0.1424$ & 
$\pm 0.0703$ &  &  &  \\ 
\bottomrule &  &  &  &  &  &  &  &  &  & 
\end{tabular}
}
\caption{Zeros of $Q^{1}_4(x)$ and $u(x,1;4)$ for some values of $M$.}
\label{Tabla4}
\end{table}



\begin{table}[ht]
\centering
\renewcommand{\arraystretch}{1.3} {\footnotesize 
\begin{tabular}{@{}rrrrrrrrrrr}
\toprule & \multicolumn{7}{c}{$t = 1.5$} & \phantom{} &  &  \\ 
\cmidrule{2-8} & $M=0$ & $M=0.002$ & $M=0.05$ & $M=0.5$ & $M=5$ & $M=10$ & $%
M=50$ &  &  &  \\ 
\midrule $Q^{t}_4(x)$ &  &  &  &  &  &  &  &  &  &  \\ 
& $\pm$1.4485 & $\pm$1.4484 & $\pm$1.4474 & $\pm$1.4395 & $\pm$1.4120 & $\pm$%
1.4047 & $\pm1.3964$ &  &  &  \\ 
& $\pm$0.8059 & $\pm$0.8057 & $\pm$0.8010 & $\pm$0.7603 & $\pm$0.5363 & $\pm$%
0.4290 & $\pm 0.2176$ &  &  &  \\ 
$u(x,t;4)$ &  &  &  &  &  &  &  &  &  &  \\ 
& $\pm0.5175i$ & $\pm0.5179i$ & $\pm0.5263i$ & $\pm0.5714i$ & $\pm0.663i$ & $%
\pm0.6899i$ & $\pm0.7282 i$ &  &  &  \\ 
& $0$ & $\pm0.0233$ & $\pm$0.1125 & $\pm0.2770$ & $\pm0.3021$ & $\pm0.2477$
& $\pm0.1261$ &  &  &  \\ 
\bottomrule &  &  &  &  &  &  &  &  &  & 
\end{tabular}
}
\caption{Zeros of $Q^{1.5}_4(x)$ and $u(x,1.5;4)$ for some values of $M$.}
\label{Tabla5}
\end{table}



\begin{table}[ht]
\centering
\renewcommand{\arraystretch}{1.3} {\footnotesize 
\begin{tabular}{@{}rrrrrrrrrrr}
\toprule & \multicolumn{7}{c}{$t = 2$} & \phantom{} &  &  \\ 
\cmidrule{2-8} & $M=0$ & $M=0.002$ & $M=0.05$ & $M=0.5$ & $M=5$ & $M=10$ & $%
M=50$ &  &  &  \\ 
\midrule $Q^{t}_4(x)$ &  &  &  &  &  &  &  &  &  &  \\ 
& $\pm$1.60437 & $\pm$1.6043 & $\pm$1.6038 & $\pm$1.5989 & $\pm$1.5717 & $%
\pm $1.5594 & $\pm1.5406$ &  &  &  \\ 
& $\pm$1.0429 & $\pm$1.0428 & $\pm$1.0408 & $\pm$1.0220 & $\pm$0.8736 & $\pm$%
0.7644 & $\pm 0.4490$ &  &  &  \\ 
$u(x,t;4)$ &  &  &  &  &  &  &  &  &  &  \\ 
& $0$ & $\pm0.0605i$ & $\pm0.1906i$ & $\pm0.3316i$ & $\pm0.4564i$ & $%
\pm0.4885i$ & $\pm0.5748 i$ &  &  &  \\ 
& $\pm0.1487$ & $\pm0.1613$ & $\pm$0.2539 & $\pm0.4256$ & $\pm0.5625$ & $%
\pm0.5117$ & $\pm0.2832$ &  &  &  \\ 
\bottomrule &  &  &  &  &  &  &  &  &  & 
\end{tabular}
}
\caption{Zeros of $Q^{2}_4(x)$ and $u(x,2;4)$ for some values of $M$.}
\label{Tabla6}
\end{table}



\begin{table}[ht]
\centering
\renewcommand{\arraystretch}{1.3} {\footnotesize 
\begin{tabular}{@{}rrrrrrrrrrr}
\toprule & \multicolumn{7}{c}{$t = 1$} & \phantom{} &  &  \\ 
\cmidrule{2-8} & $M=0$ & $M=0.002$ & $M=0.05$ & $M=0.5$ & $M=5$ & $M=10$ & $%
M=50$ &  &  &  \\ 
\midrule $Q^{t}_6(x)$ &  &  &  &  &  &  &  &  &  &  \\ 
& $\pm$1.51614 & $\pm$1.51612 & $\pm$1.5153 & $\pm$1.5103 & $\pm$1.5018 & $%
\pm $1.5005 & $\pm1.4993$ &  &  &  \\ 
& $\pm$1.0730 & $\pm$1.0729 & $\pm$1.0711 & $\pm$1.0600 & $\pm$1.0403 & $\pm 
$1.0374 & $\pm1.0346$ &  &  &  \\ 
& $\pm$0.4530 & $\pm$0.4526 & $\pm$ 0.4445 & $\pm$0.3846 & $\pm$0.2044 & $%
\pm $0.1524 & $\pm0.0714$ &  &  &  \\ 
$u(x,t;6)$ &  &  &  &  &  &  &  &  &  &  \\ 
& $\pm0.9164i$ & $\pm0.9165i$ & $\pm0.9185i$ & $\pm0.9300i$ & $\pm0.9501i$ & 
$\pm0.9533i$ & $\pm0.9564 i$ &  &  &  \\ 
& $\pm0$ & $\pm0.0139$ & $\pm$0.0665 & $\pm0.1444$ & $\pm0.1109$ & $%
\pm0.0852 $ & $\pm0.0409$ &  &  &  \\ 
\bottomrule &  &  &  &  &  &  &  &  &  & 
\end{tabular}
}
\caption{Zeros of $Q^{1}_6(x)$ and $u(x,1;6)$ for some values of $M$.}
\label{Tabla7}
\end{table}



\begin{table}[ht]
\centering
\renewcommand{\arraystretch}{1.3} {\footnotesize 
\begin{tabular}{@{}rrrrrrrrrrr}
\toprule & \multicolumn{7}{c}{$t = 1$} & \phantom{} &  &  \\ 
\cmidrule{2-8} & $M=0$ & $M=0.002$ & $M=0.05$ & $M=0.5$ & $M=5$ & $M=10$ & $%
M=50$ &  &  &  \\ 
\midrule $Q^{t}_{10}(x)$ &  &  &  &  &  &  &  &  &  &  \\ 
& $\pm$1.79469 & $\pm$1.79467 & $\pm$1.7942 & $\pm$1.7921 & $\pm$1.7896 & $%
\pm $1.7893 & $\pm1.7890$ &  &  &  \\ 
& $\pm$1.49286 & $\pm$1.49284 & $\pm$1.4922 & $\pm$1.4888 & $\pm$1.4849 & $%
\pm $1.4845 & $\pm1.4841$ &  &  &  \\ 
& $\pm$1.17419 & $\pm$1.17414 & $\pm$1.1730 & $\pm$1.1674 & $\pm$1.1608 & $%
\pm $1.1600 & $\pm1.1593$ &  &  &  \\ 
& $\pm$0.7931 & $\pm$0.7930 & $\pm$0.7907 & $\pm$0.7789 & $\pm$0.7647 & $\pm 
$0.7630 & $\pm0.7616$ &  &  &  \\ 
& $\pm$0.2950 & $\pm$0.2947 & $\pm$ 0.2858 & $\pm$0.2291 & $\pm$0.1067 & $%
\pm $0.0780 & $\pm0.0359$ &  &  &  \\ 
$u(x,t;10)$ &  &  &  &  &  &  &  &  &  &  \\ 
& $\pm1.1107i$ & $\pm1.1108i$ & $\pm1.1117i$ & $\pm1.1166i$ & $\pm1.1223i$ & 
$\pm1.1230i$ & $\pm1.1236 i$ &  &  &  \\ 
& $\pm0$ & $\pm0.0110$ & $\pm$0.0509 & $\pm0.0950$ & $\pm0.0589$ & $%
\pm0.0440 $ & $\pm0.0206$ &  &  &  \\ 
\bottomrule &  &  &  &  &  &  &  &  &  & 
\end{tabular}
}
\caption{Zeros of $Q^{1}_{10}(x)$ and $u(x,1;10)$ for some values of $M$.}
\label{Tabla8}
\end{table}



\end{document}